\documentclass[11pt,a4paper]{amsart}

\usepackage[UKenglish]{babel}
\usepackage[utf8x]{inputenc}
\usepackage[T1]{fontenc}
\usepackage{stmaryrd,verbatim}
\usepackage{times}
\usepackage[pagewise]{lineno}
\usepackage{changepage}

\usepackage[a4paper,top=2cm,bottom=2cm,left=2cm,right=2cm,marginparwidth=1.75cm]{geometry}

\usepackage{amsmath,amssymb,amsfonts,amsthm,upgreek}
\usepackage{mathrsfs}
\usepackage{graphicx}
\usepackage[pdftex,dvipsnames]{xcolor}
\usepackage[colorinlistoftodos,prependcaption, textsize=tiny, textwidth=2.5cm]{todonotes}
\usepackage[all]{xy}
\usepackage{cancel}
\usepackage{appendix}
\usepackage{enumerate}
\usepackage{multicol}
\usepackage{xargs}

\usepackage[colorlinks=true, allcolors=blue,backref=page]{hyperref}
\usepackage[alphabetic]{amsrefs}

\newcommand{\ra}{{\rm a}}
\newcommand{\rb}{{\rm b}}
\newcommand{\rc}{{\rm c}}
\newcommand{\rd}{{\rm d}}
\newcommand{\re}{{\rm e}}

\newcommand{\rp}{{\rm p}}
\newcommand{\rmq}{{\rm q}}

\newcommand{\rv}{{\rm v}}

\newcommand{\rF}{{\rm F}}
\newcommand{\rG}{{\rm G}}
\newcommand{\rH}{{\rm H}}

\newcommand{\rL}{{\rm L}}
\newcommand{\rM}{{\rm M}}

\newcommand{\rP}{{\rm P}}

\newcommand{\rT}{{\rm T}}

\newcommand{\bi}{{\bf i}}

\newcommand{\bN}{{\bf N}}

\newcommand{\cB}{\mathcal{B}}

\newcommand{\cD}{\mathcal{D}}
\newcommand{\cE}{\mathcal{E}}
\newcommand{\cF}{\mathcal{F}}

\newcommand{\cH}{\mathcal{H}}
\newcommand{\cI}{\mathcal{I}}
\newcommand{\cJ}{\mathcal{J}}

\newcommand{\cL}{\mathcal{L}}
\newcommand{\cM}{\mathcal{M}}

\newcommand{\cT}{\mathcal{T}}

\newcommand{\cV}{\mathcal{V}}

\newcommand{\sF}{\mathscr{F}}

\newcommand{\sH}{\mathscr{H}}

\newcommand{\sR}{\mathscr{R}}

\newcommand{\fg}{{\mathfrak g}}

\newcommand{\fh}{{\mathfrak h}}

\newcommand{\fm}{{\mathfrak m}}

\newcommand{\fu}{{\mathfrak u}}

\newcommand{\fF}{{\mathfrak F}}

\newcommand{\rII}{{\rm II}}

\newcommand{\N}{\bN}

\newcommand{\R}{\mathbb{R}}
\newcommand{\C}{\mathbb{C}}

\newcommand{\so}{\mathfrak{so}}

\newcommand{\gl}{\mathfrak{gl}}

\newcommand{\SO}{{\rm SO}}
\newcommand{\Sp}{{\rm Sp}}
\newcommand{\Spin}{{\rm Spin}}
\newcommand{\SU}{{\rm SU}}
\newcommand{\GL}{\mathrm{GL}}

\newcommand{\U}{{\rm U}}

\renewcommand{\P}{\bP}

\newcommand{\Ad}{\mathrm{Ad}}

\newcommand{\Crit}{\mathrm{Crit}}

\newcommand{\End}{{\mathrm{End}}}

\renewcommand{\epsilon}{\varepsilon}

\newcommand{\Hol}{\mathrm{Hol}}
\newcommand{\Hom}{{\mathrm{Hom}}}

\newcommand{\Lie}{\mathrm{Lie}}
\newcommand{\Ric}{{\rm Ric}}

\newcommand{\id}{\mathrm{id}}
\newcommand{\im}{\mathop{\mathrm{im}}}

\newcommand{\pr}{\mathrm{pr}}
\newcommand{\rank}{\mathop{\mathrm{rank}}}

\newcommand{\tr}{\mathop{\mathrm{tr}}\nolimits}

\newcommand{\vol}{\mathrm{vol}}

\newcommand{\Rep}{\mathrm{Rep}}

\newcommand{\Skew}{\mathrm{Skew}}

\newcommand{\Fr}{{\rm Fr}}

\newcommand{\Proj}{\mathrm{Proj}}

\newcommand{\sn}{\mathbb{S}}

\newcommand{\pg}{\boldsymbol\Phi}
\newcommand{\zg}{\boldsymbol\xi}
\newcommand{\dg}{\boldsymbol{\mathcal{F}}}

\newcommand{\qandq}{\quad\text{and}\quad}

\newcommand{\co}{\mskip0.5mu\colon\thinspace}
\def\<{\mathopen{}\left<}
\def\>{\right>\mathclose{}}
\def\({\mathopen{}\left(}
\def\){\right)\mathclose{}}
    
\def\wtilde{\widetilde}

\usepackage{multicol, color}

\definecolor{gold}{rgb}{0.85,.66,0}
\definecolor{cherry}{rgb}{0.9,.1,.2}
\definecolor{burgundy}{rgb}{0.8,.2,.2}
\definecolor{orangered}{rgb}{0.85,.3,0}
\definecolor{orange}{rgb}{0.85,.4,0}
\definecolor{olive}{rgb}{.45,.4,0}
\definecolor{lime}{rgb}{.6,.9,0}
\definecolor{green}{rgb}{.2,.7,0}
\definecolor{grey}{rgb}{.4,.4,.2}
\definecolor{brown}{rgb}{.4,.3,.1}

\newtheorem{theorem}{Theorem}
\newtheorem{proposition}{Proposition}
\newtheorem{corollary}[proposition]{Corollary}
\newtheorem{lemma}[proposition]{Lemma}

\theoremstyle{remark}
\newtheorem{remark}[proposition]{Remark}

\theoremstyle{definition}
\newtheorem{definition}[proposition]{Definition}
\newtheorem{example}[proposition]{Example}

\newcommand{\ufm}{{\underline{\fm}}}
\newcommand{\ufg}{{\underline{\fg}}}
\newcommand{\ufh}{{\underline{\fh}}}
\newcommand{\uso}{\underline{\so}}
\newcommand{\txi}{\tilde{\Xi}}

\newcommand{\Div}{\operatorname{div}}

\newcommand{\qwithq}{\quad\text{with}\quad}
\newcommand{\qforq}{\quad\text{for}\quad}

\newcommand{\Stab}{\operatorname{Stab}}

\newcommand{\lan}{\langle}
\newcommand{\ran}{\rangle}

\newcommand{\di}{\operatorname{div}}
\newcommandx{\change}[2][1=]{\todo[linecolor=blue,backgroundcolor=blue!25,bordercolor=blue,#1]{#2}}

\newcommand{\intprod}{%
    \mathbin{\scalebox{1.5}{$\lrcorner$}}%
    }

\renewcommand{\P}{P}

\addto\captionsUKenglish{}

\title{Harmonic flow of geometric structures}
\author{Eric Loubeau \& Henrique N. S\'a Earp}

\address[Eric Loubeau]{Univ. Brest, CNRS UMR 6205, LMBA, F-29238 Brest, France}
\email{loubeau@univ-brest.fr}
\address[Henrique N. S\'a Earp]{University of Campinas (Unicamp), Brazil}
\email{henrique.saearp@ime.unicamp.br}

\begin{document}

\begin{abstract}
We give a twistorial interpretation of geometric structures on a Riemannian manifold, as sections  of homogeneous fibre bundles, following an original insight by Wood (2003). The natural Dirichlet energy induces an abstract harmonicity condition, which gives rise to a geometric gradient flow. We establish a number of analytic properties for this flow, such as uniqueness, smoothness, short-time existence, and some sufficient conditions for long-time existence. This description potentially subsumes a large class of geometric PDE problems from different contexts.

As applications, we recover and unify a number of results in the literature: for the isometric flow of $\rG_2$-structures, by Grigorian (2017, 2019), Bagaglini (2019), and Dwivedi-Gianniotis-Karigiannis (2019); and for harmonic almost complex structures, by He (2019) and He-Li (2019). Our theory also establishes original properties regarding harmonic flows of parallelisms and almost contact structures.
\end{abstract}

\maketitle
\vspace{-1.0cm}
\begin{adjustwidth}{0.95cm}{0.95cm}
    \tableofcontents
\end{adjustwidth}


\newpage
\section*{Introduction}

We formulate a general theory of harmonicity for geometric structures on an oriented Riemannian manifold  $(M^m,g)$, with structure group $G\subset \SO(m)$, building upon a framework originally outlined in \cites{Wood1997,Wood2003}, and further considered  eg. in \cites{Gonzalez-Davila2009}. 
A \emph{geometric structure} for our purposes is a smooth section of a tensor bundle $\cF\subset\cT^\bullet(M)$, typically stabilised by a subgroup $H\subset G$, eg. \begin{itemize}
    \item 
    almost complex structures, for $\U(n)\subset \SO(2n)$;
    
    \item
    almost contact structures, for $\U(n)\subset\SO(2n+1)$;
    
    \item
    $\rG_2$--structures, for $\rG_2\subset\SO(7)$;
    
    \item
    $\Spin(7)$-structures, for $\Spin(7)\subset\SO(8)$ etc.
\end{itemize}
A geometric structure $\xi$ can be viewed as a section of the homogeneous fibre bundle $\pi:N:=P/H\to M$, which emerges by reduction of the (oriented) frame bundle $P\to M$, under a one-to-one correspondence [see \eqref{eq: Xi correspondence}]:
$$
\{\xi:M\to \cF\} 
\quad\leftrightarrow\quad 
\{\sigma:M\to N\}.
$$
We assume throughout that $M$ is closed, although the theory can be easily extended to compactly supported sections over complete manifolds with bounded geometry.
 
Given a suitable fibre metric $\eta$ on $N$ (see \eqref{eq: conditions of HSF} below), a natural Dirichlet energy can be assigned to such sections  $\sigma\in\Gamma(\pi)$: 
$$
E(\sigma):=\frac{1}{2}\int_M |d^\cV\sigma|_\eta^2,
$$
where the \emph{vertical torsion} $d^\cV\sigma$ is the projection of $d\sigma$ onto the distribution $\ker\pi_*$, tangent to the fibres of $\pi:N\to M$. A geometric structure is defined to be \emph{harmonic} if $\sigma\in\Crit({E})$, and \emph{torsion-free} if $d^\cV\sigma=0$. 
The critical set is the vanishing locus of the \emph{vertical tension} field 
$\tau^\cV(\sigma):=\tr_g\nabla^{\cV} d^\cV\sigma$, where $\nabla^{\cV}$ is the (pull-back by $\sigma$ of the) vertical part of the Levi-Civita connection of $(N,\eta)$. 

Just as in the general theory of harmonic maps, the notion of harmonic geometric structure depends of course on the fixed Riemannian metric. The explicit form of $\tau^\cV$, in each particular context, then defines a natural geometric PDE. Such condition is typically weaker than the corresponding notions of `integrability' and `torsion-freedom', thus in favourable cases harmonicity can still characterise the `best' geometric structure, even when stronger conditions are otherwise obstructed or trivial. In Section \ref{sec: method to determine tau^V},  we provide a general method to determine the actual harmonicity condition in any given context, which subsumes several cases previously studied in the literature, eg. the harmonicity of almost complex structures \cite{Wood1993} and almost contact structures  \cite{Vergara-Diaz2006}, and also the most recent $\Div T=0$ condition in $\rG_2$-geometry, originally found by \cite{Grigorian2017} from a rather different perspective. We also propose an original formulation for the problem of harmonic parallelisms eg. on spheres.

Starting from any smooth section $\sigma_0\in \Gamma(\pi)$  of the homogeneous bundle $\pi:N\to M$, the Dirichlet energy gives rise to a natural gradient flow, which we will call the \emph{harmonic section flow}:
\begin{equation}
\label{eq: HSF}
\tag{HSF}
    \begin{cases}
    \displaystyle \partial_t \sigma_t &= \tau^\cV (\sigma_t)\\
    \sigma_t|_{t=0} &= \sigma_0 \in \Gamma(\pi)
    \end{cases},
\quad\text{on}\quad 
    M_{\rT}:=M\times \left[0,\rT\right[. 
\end{equation}
Defining the flow in terms of the vertical tension $\tau^\cV$ guarantees precisely that a solution $\{\sigma_t\}$ flows \emph{among sections}, as opposed to just maps from $M$ to $N$ (Proposition~\ref{prop HMF for sections}). The main purpose of this paper is to initiate an analytic theory for \eqref{eq: HSF}. We establish a priori estimates, uniqueness and short-time existence, and sufficient conditions for long-time existence, regularity and convergence of solutions, which hold regardless of the specific geometric context, assuming the following data:
\begin{equation}
\label{eq: conditions of HSF}
\tag{A}
\text{\parbox{.90\textwidth}
    {$\bullet$ $G\subset\SO(m)$ a compact semi-simple Lie group;\\
    $\bullet$ $H\subset G$ a normal reductive Lie subgroup;\\
    $\bullet$ $P$ a principal $G$-bundle over a closed Riemannian manifold $(M^m,g)$;  \\
    $\bullet$ $\eta$ an $H$-invariant fibre metric on $P$, constructed from a compatible bi-invariant metric on $G$. 
}}
\end{equation}

Our approach consists of exploiting, as far as possible, classical results and techniques form harmonic map theory, most notably by means of a further one-to-one correspondence [cf. \S \ref{sec: sigma <-> s}] between sections $\sigma:M\to N$ and their $G$--equivariant lifts $s:P\to G/H$, defined on the total space of the $G$-bundle $P$ (typically taken to be the oriented frame bundle of $M$). Crucially, under the assumptions \eqref{eq: conditions of HSF}, the harmonic section flow is in a certain sense equivalent to the harmonic map heat flow for $G$--equivariant lifts [Proposition \ref{prop: HMHF}]. In \S\ref{section correspondence}, we establish:

\begin{theorem}
\label{thm: uniqueness and s-t existence}
Under the assumptions \eqref{eq: conditions of HSF}, for each $\sigma_0 \in \Gamma(\pi)$, there exists a maximal time 
    \mbox{$0 < \rT_{\max}(\sigma_0)
    \leq \infty$} 
    such that \eqref{eq: HSF} admits a unique smooth (short-time) solution $\{\sigma_t\}$ on $M_{\rT_{\max}}$.
\end{theorem}

In \S\ref{sec: long-time existence}, we prove the following main results.
Regarding long-time existence and convergence to a harmonic limit, we obtain an upper estimate for  the finite-time blow-up rate and $C^0$-blow-up of density:
\begin{theorem}
\label{thm: Tmax and doubling time}
    Under the assumptions \eqref{eq: conditions of HSF}, let $\{\sigma_t\}$ be a solution of \eqref{eq: HSF} on $M_\rT$, for some $0<\rT\leq\infty$, and set 
    $$
    \epsilon(t):=\frac{1}{2}\vert d^\cV \sigma_t\vert_\eta^2
    \qandq
    \bar{\epsilon}_t:= \sup_M\epsilon(t).
    $$
\begin{enumerate}[(i)]
\item 
    There exist $C(M,g)>0$ and $0<\delta \leq \min\{\rT,\frac{1}{C\bar\epsilon_0}\}$  such that
    $$
    \epsilon(t) 
    \leq \bar{\epsilon}_0 \frac{1}{ 1 - C\bar{\epsilon}_0 t},
    \quad\forall t \leq \delta.
    $$
    Moreover, $C>0$ can be chosen so that
    $$
    \epsilon(t) 
    \leq 2 \bar\epsilon_0,
    \quad\forall t \leq \delta.
    $$
    Conversely, for every $\rT>0$, there exists $\gamma(\rT) >0$ such that, for any initial condition satisfying $\bar\epsilon_0 <\gamma$, the flow exists on $M_\rT$.
\medskip
\item  
    If $\{\sigma_t\}$ cannot be extended beyond some $\rT_{{\rm max}}<\infty$, then 
    $
    \displaystyle\lim_{t \nearrow \rT_{\max}} \bar\epsilon_t 
    = \infty.
    $
\end{enumerate} 
\end{theorem}
Moreover, the bounded torsion condition, which is sufficient for long-time existence by Theorem \ref{thm: Tmax and doubling time}--(ii), can be slightly weakened into a time-uniform $L^p$-bound, for large enough $p\geq m:=\dim M$: 
\begin{theorem}
\label{thm: HSF under bounded torsion}
    Under the assumptions \eqref{eq: conditions of HSF}, if a solution $\{\sigma_t\}$ of  \eqref{eq: HSF}, with any initial condition $\sigma_0\in\Gamma(\pi)$, has  $\epsilon(t)=|d^\cV\sigma_t|^2$ bounded in $L^m(M)$, uniformly in $t$, then 
    there exists a sequence $t_k\nearrow\infty$ along which $\{\sigma_t\}$  converges to a smooth section $\sigma_\infty\in\Gamma(\pi)$, defining a harmonic geometric structure.
\end{theorem}

At this point one might expect that, if the initial energy $E(\sigma_0)$ is sufficiently small, the flow might actually converge to a torsion-free geometric structure. A preliminary result can indeed be established in that sense, based on the standard theory of harmonic map heat flows [Proposition \ref{prop: Chen-Ding small energy}], but its applicability is extremely limited, cf. Remark \ref{rem: unfortunate Chen-Ding}. Moreover, since finding torsion-free geometric structures is in general quite hard, one should expect initial conditions with small energy to be just as difficult to arrange.

An important motivation for this analytic theory stems from recent developments in $\rG_2$--geometry on $7$-dimensional manifolds. In this context, a \emph{$\rG_2$--structure} is a section $\xi=\varphi$ of the bundle of positive $3$--forms $\cF=\Omega_+^3(M)$, which induces a \emph{  $\rG_2$--metric}  $g_\varphi$, and its \emph{full torsion tensor} $T:=\nabla^{g_\varphi}\varphi$ is essentially the same as the vertical torsion $d^\cV\sigma$. Distinct  $\rG_2$--structures are \emph{isometric} if they yield the same  $\rG_2$--metric, and the recent article  \cite{Grigorian2017} interprets the \emph{divergence-free torsion} condition    $\Div T=0$  as a `gauge-fixing' among isometric $\rG_2$--structures, by viewing $\varphi$ as a connection on a certain octonionic bundle. The natural geometric flow  $\dot\varphi=\Div T\lrcorner (\ast_{g_\varphi}\varphi)$ associated to this condition has since attracted substantial interest: short-term existence and uniqueness were first established in \cite{Bagaglini2019}, and sufficient conditions for long-time existence and regularity were studied independently by \cites{Grigorian2019,Dwivedi2019}.  
We offer an alternative formulation of this equation as a harmonicity condition on sections of the  $\R P^7$--bundle $N=P_{\SO(7)}/\rG_2$. A number of  results in the above literature can then be seen as instances of Theorems \ref{thm: uniqueness and s-t existence}-\ref{thm: HSF under bounded torsion}, cf. \S\ref{sec: Results for divT-flow}]. Moreover, Theorems \ref{thm: Tmax and doubling time} and \ref{thm: HSF under bounded torsion} offer a new line of attack to the problem of constructing such solutions on concrete cases, most notably opening a pathway for mass-producing examples over homogeneous manifolds, along the lines of \cite{Lauret2016}, see Afterword.

In addition, Theorems \ref{thm: uniqueness and s-t existence}-\ref{thm: HSF under bounded torsion} generalise some results from a similar recent study by \cites{He2019,He2019a} in the context of harmonic almost complex structures, i.e. sections $\xi=J$ of the subbundle $\cF\subset \End(TM)$ squaring to minus the identity. Then harmonicity is equivalent to the commutation $[\nabla^ *\nabla J,J]=0$, and the corresponding geometric flow is shown to always exist at short-time, then to blow-up at controlled rates exactly as prescribed by our Theorems, cf. Corollary \ref{cor: bounded T => T-> infty ACS}. It should be noted that, just like in the $\rG_2$ case, the authors are able to achieve analytically much more than our general theory, by mobilising algebraic and differential properties which are contingent to each particular context.

Our results seem to be hitherto unknown  in the cases of parallelisms, when $H=\{e\}\subset\SO(n)$ (Corollary \ref{cor: bounded T => T-> infty S^ 3}), and of almost contact structures, when $H=\U(n)\subset\SO(2n+1)$ (Corollary \ref{cor: bounded T => T-> infty ACTCS}). Therefore in those contexts our conclusions are completely original. Finally, since this text first appeared as a preprint, further studies of harmonicity appeared in the cases of $\Spin(7)\subset\SO(8)$ \cite{Dwivedi2021} and of homogeneous $\rG_2$-structures on the sphere $S^7=\Sp(2)/\Sp(1)$ \cite{Loubeau2022}.

\bigskip
\noindent
\textbf{Acknowledgements:} This project started during a visiting chair by EL, sponsored by Unicamp and Institut Français du Brésil, in 2017. The present article stems from an ongoing CAPES-COFECUB bilateral collaboration (2018-2021), granted by Brazilian Coordination for the Improvement of Higher Education Personnel (CAPES) – Finance Code 001 [88881.143017/2017-01], and Campus France [MA 898/18]. HSE has been funded by São Paulo Research Foundation (Fapesp)  \mbox{[2017/20007-0]} \& \mbox{[2018/21391-1]} and the National Council for Scientific and Technological Development (CNPq)  \mbox{[307217/2017-5]}.

\part{Harmonicity theory of geometric structures}
\section{Universal sections on homogeneous fibre bundles}\label{Section 1}
Let $G$ be a Lie group and $p:\P\to M$ a principal $G$-bundle. Given a Lie subgroup $H\subseteq G$, denote by $N:=\P/H$ its orbit space, by $q: \P\to N$ the corresponding principal $H$-bundle and by $\pi : N \to M$ the projection on the quotient, so that $p = \pi\circ q$. Denoting by $\fg:=\Lie(G)$ and $\fh:=\Lie(H)$, assume $H$ is naturally reductive, that is, there is an orthogonal complement $\fm$ satisfying
$$
\fg=\fh\oplus\fm 
\qandq 
\Ad_G(H)\fm\subseteq\fm.
$$

\subsection{Canonical geometry on homogeneous fibre bundles}
\label{sec: canonical geom on hfb}

Consider a connection $\omega\in\Omega^1(\P,\fg)$  on $P$ (e.g. if $\P=\Fr(M)$), inducing the splitting
$$
TN=\cV \oplus \cH
$$
with
$
\cV:=\ker\pi_*=q_*(\ker p_*) 
$
and
$
\cH:=q_*(\ker\omega).
$
\begin{multicols}{2}\label{figure 1}
Let $\ufm\to N$ be the vector bundle associated to $q$ with fibre $\fm$. Its points are the $H$-equivalence classes of `vectors-in-a-frame':
$$
z\bullet w:=[(z,w)]_{H}=\cI(q_*(w_z^*))
$$
with
$$
w_z^*:= \left.\partial_t\right\vert_{t=0} z. \exp tw \in T_z\P,
$$
for $z\in \P$ and $w\in \fm$.
\columnbreak
$$\xymatrix{
\llap{$z\in$ } \P 
	\ar[dr]^q 
    \ar[dd]_p
    & & {\ufm} \ar[dl] \rlap{ $:=\P\times_H \fm \ni  z\bullet w $}\\ 
  &\llap{$y\in$ } N 
  \ar[dl]_\pi 
  	& \ar[l]\cV
  	\ar[u] ^\cI  \\ 
\llap{$x\in$ }  M 
& & }$$
\end{multicols}
Then we have a vector bundle isomorphism
\begin{equation}
\label{eq: isomorphism I:V->m}
\begin{array}{rcccl}
    \cI&:&\cV&\tilde{\rightarrow}&\ufm\\
    &&q_*(w_z^*)&\mapsto&z\bullet w
\end{array}.
\end{equation}
\begin{multicols}{2}
The $\fm$-component $\omega_\fm \in \Omega^1(\P,\fm)$ of the connection is $H$-equivariant and $q$-horizontal, so it projects to a \emph{homogeneous connection form} $f \in\Omega^1(N,\ufm)$ defined by:
\begin{equation}
\label{eq: homog connection form f}
    f (q_*(Z)):= z \bullet \omega_\fm(Z) 
    \qforq
    Z\in T_z\P.
\end{equation}
On $\pi$-vertical vectors $f $ coincides with the canonical isomorphism (\ref{eq: isomorphism I:V->m}), while $\pi$-horizontal vectors are in the kernel: 
$$
f (v_y)=\cI(v_y^\cV), 
\qforq
v_y\in T_yN.
$$
\columnbreak
$$\xymatrix{
\llap{$v_z\in$ }T\P 
	\ar[dr]^{q_*} 
    \ar[dd]_{p_*} 
    \ar[r]^\omega 
    \ar@/^2pc/[r]^{\omega_\fm}
    & \fg = \fm \ar@/^2pc/[r]^{z\bullet} \oplus \fh & {\ufm} \\ 
  &TN 
  	\ar[dl]_{\pi_*} 
    \ar@{=}[r]
    \ar[ur]_f 
    & \cV \rlap{ $\oplus\;\cH$}
  		\ar[u] ^\cI  \\ 
TM & &     
}$$
\end{multicols}

\subsection{The universal section of a geometric structure}\label{Section 1.2}

A vector bundle $\cF\to M$ will be called \emph{geometric} (with respect to $\P\to M$)  if there exists a \emph{geometric representation} $\rho\in\Rep(\P,\cF)$, i.e., a monomorphism of principal bundles 
$$
\rho: \P\hookrightarrow \Fr(\cF).
$$
For simplicity, we may assume $\cF \subset \bigoplus T^{p,q}$ is a tensor bundle.
Denote by $V=\R^r$ the typical fibre of $\cF$, with $r=\rank(\cF)$. If $\cF$ is geometric, then, at each point $x\in M$, the map $\rho$ identifies the element $z_x\in \P_x$ with a frame of $\cF_x$, i.e., with a linear isomorphism 
$$
\rho(z_x):\cF_x \tilde\to \;V.
$$
Fixing an element $\xi_0\in V$, a \emph{geometric structure modeled on $\xi_0$} is a section $\xi\in\Gamma(\cF)$ such that, for each $x\in M$, the induced map $\rho: \P_x\to V$ is surjective at $\xi_0$, i.e. for any $x\in M$, there is always a frame of $T_x M$ whose image by $\rho$ is a frame of $\cF_x$ identifying $\xi(x)$ and $\xi_0$. For simplicity, let us just think henceforth of $\rho$ as a fibrewise element of $\Hom(G,\GL(V))$ and omit explicit mention of it. 

For each $z\in \P$, the frame $\rho(z)$ defines an isomorphism $\cF_{p(z)} \simeq V$, and the right-action of $G$ on $\P$ then carries over to $V$, in the following way.
For $u\in V$, consider $b\in\cF_x$  
the vector of coordinates $u$  in the frame $\rho(z)$. Then define $g.u \in V$ to be the coordinates of $b\in \cF_x$ in the frame $\rho(z.g^{-1})$.
This action is linear, so we can represent $g\in G$ by a matrix $\rM_g\in \GL(V)$, such that $g.u = \rM_g(u)$. Differentiating at the identity, we obtain the induced Lie algebra action of $a\in \fg$ on $V$, which we can also identify with the action of a matrix $A \in  \gl (V)$: $a.u = Au$. 
Then
\begin{align*}
    g. (a.u) &= \rM_g (A u) = (\rM_gA\rM_g^{-1}) (\rM_gu) 
    = ((\Ad_{G} g)a) . (g.u)\\
    &= (g.a) . (g.u).
\end{align*}
Suppose now $H\subseteq G$ fixes the model structure $\xi_0$:
\begin{equation}
\label{eq: H=Stab(xi)}
    H=\Stab(\xi_0).    
\end{equation}
In view of \eqref{eq: H=Stab(xi)}, a \emph{universal section} $\Xi\in\Gamma(N,\pi^*(\cF))$ is well-defined by
\begin{align}
\label{eq: universal section}
    \Xi(y):=y^*\xi_0.
\end{align}
Explicitly, one assigns to $y\in N$ the vector of $\cF_{\pi(y)}$ whose coordinates are given by $\xi_0$ in any frame $\rho(z_{\pi(y)})$.
Now each section $\sigma\in\Gamma(M,N)$ induces a  geometric structure $\xi_\sigma\in\Gamma(M,\cF)$ modelled on $\xi_0$:
\begin{align}
\label{eq: Xi correspondence}
    \xi_\sigma:=\sigma^*\Xi
    =\Xi\circ \sigma.
\end{align}
Since $\pi^*(\cF)$ is isomorphic to the associated bundle $\pi^*\P \times_G V$, there exists a $G$-equivariant map $\tilde{\Xi} : \pi^*\P \to V$ such that
$$
\tilde{\Xi} = \rho\circ\Xi\circ \pi^*p .
$$
\begin{multicols}{2}
This map $\tilde{\Xi}$ associates to $(z,y)\in \pi^*\P$ the coordinates of the vector $\Xi(y)\in \cF_{\pi(y)}$, but in a frame $\rho(z)$, for $z$ not necessarily in the equivalence class $q^{-1}(y)$. 
Note that $\txi|_\P = \xi_0$, by the $H$-equivariant embedding 
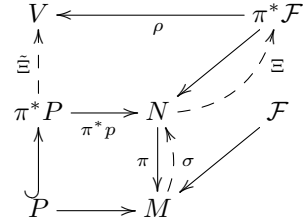
\begin{align*}
    \P &\hookrightarrow \pi^*\P\\
    z &\mapsto (z,q(z)).
\end{align*}
\columnbreak
$$\xymatrix{
     V 
	&&\pi^*\cF \ar[dl]\ar@{^{}->} [ll]^{\rho}
    \\
\pi^*\P 
	\ar[r]_{\pi^*p}
    \ar@{-->} [u]^{\txi}
    &N \ar[d]_\pi \ar@{-->}@/_1pc/[ur]_(0.75)\Xi& \cF \ar[dl]\\
\P \ar@{^{(}->}[u] \ar[r]&M \ar@{-->}@/_/[u]_\sigma &
}$$
\end{multicols}
\vspace{-0.8cm}
\begin{lemma}
\label{diff}
    Let  $p:\P\to M$ be a principal $G$-bundle, $H\subset G$ a naturally reductive Lie subgroup with $\fg=\fh\oplus\fm$, $\pi: N=\P/H \to M$ its orbit space and $q: \P\to N$ the corresponding principal $H$-bundle.
    Let $f \in\Omega^1(N,\ufm)$ be the homogeneous connection form associated to the connection $\omega\in\Omega^1(\P,\fg)$ and the splitting $TN=\cV \oplus \cH$:
\begin{equation*}
    f (q_*(Z))= z \bullet \omega_\fm(Z) 
    \qforq
    Z\in T_z\P , z\in \P.
\end{equation*} 
    Let $\cF\to M$ be a geometric vector bundle, such that $H$ fixes an element $\xi_0$ in its typical fibre, cf. \eqref{eq: H=Stab(xi)}. 

    Then, the covariant derivative of the universal section $\Xi\in\Gamma(N,\pi^*(\cF))$, defined by \eqref{eq: universal section}, is 
\begin{eqnarray} 
\label{eq: diff}
    \nabla_Y \Xi = f (Y) . \Xi, 
    \qforq
    Y \in TN.
\end{eqnarray}
\end{lemma}
\begin{proof}
    Let $Z\in T\P$ be a lift of $Y\in TN$, i.e. $dq(Z)=Y$. Applying the exterior covariant derivative $D$ for $V$-valued differential forms on $\pi^*\P$, in the direction $Z$:
\begin{align*}
    D\txi(Z)
    &= d\txi (Z) +\omega(Z).\txi 
    =\omega(Z).\txi , \\
    &= \omega_{\fm}(Z).\xi_0 , 
\end{align*}
    since $\txi|_\P = \xi_0$ and $\fh . \xi_0=0$, by \eqref{eq: H=Stab(xi)}. 
    Under the $H$-equivariant embedding $\P \hookrightarrow \pi^* \P$, the bundle $q$ may be seen as a principal $H$-subbundle of $\pi^* p$, i.e. the $G$-action on $\pi^* \P$ restricts to an $H$-action on $q$. 
    Under the $\bullet$ operation (which we perform with the group $H$) we obtain:
    $$
    \nabla_Y \Xi = (z,q(z)) \bullet (\omega_\fm(Z) . \xi_0)
    \in \P \times_H V,
    \qforq
    Y = q_* (Z)
    \qandq 
    (z,q(z))\in \P \hookrightarrow \pi^* \P.
    $$
    Now we compute: 
\begin{align*}
    (z,q(z)) \bullet (\omega_\fm(Z) . \xi_0) 
    &= [(z,q(z))h , h^{-1}(\omega_\fm(Z) . \xi_0)]_H \\
    &= [(z,q(z))h , ((\Ad_{G} h^{-1})\omega_\fm(Z)). (h^{-1}.\xi_0)]_H \\
    &= [(z,q(z))h , (h^{-1}.\omega_\fm(Z)). \xi_0]_H \\
    &= [(z,q(z))h , (h^{-1}.\omega_\fm(Z))]_H . \Xi \\
    &= f (Y) . \Xi .
    \qedhere
\end{align*}
\end{proof}

\subsection{Determining the vertical tension field}
\label{sec: method to determine tau^V}

In any context of interest, one would like to determine the vertical tension explicitly, thus obtaining a natural geometric PDE in terms of the original objects. Here is a practical procedure to do this, mobilising the definitions of Section~\ref{sec: canonical geom on hfb}.

\subsubsection{Vertical torsion in $\ufm$}

We first use the universal structure to determine the homogeneous connection form \eqref{eq: homog connection form f}, which gives an isomorphism $f:TN\tilde{\to}\; \ufm$. Given vector fields $X\in \Gamma(TM)$ and $Y:=d\sigma (X)$, we
invert \eqref{eq: diff} to obtain an expression of $f(Y)\in\ufm$.

We can now compute the vertical tension field of a section $\sigma:M\to N$. Pulling back $f$ to $TM$, we obtain the image of the vertical torsion under $\cI:\cV\tilde{\to}\;\ufm$:
$$
    (\sigma^*f)  (X) 
    = \cI (d^\cV \sigma)(X) 
$$
since $\cI(d^\cV \sigma) = \sigma^*f  = f  (d\sigma) $. 

\subsubsection{Natural connection $\nabla^\omega$ on $\ufg$}

The connection  $\nabla^\omega$ on $\ufg$ is obtained as follows.
Let $(\cE, \langle , \rangle , \nabla)$ be an oriented Riemannian vector bundle of rank $k$ over $M$. For example, take $G = \SO(k)$ and $\P \to M$ to be the principal $\SO(k)$-frame bundle of $\cE$, so that $\cE=\P\times_{\SO(k)} \R^{k}$ and $\ufg$ is its associated bundle.

Identifying the  Lie algebra $\gl(k)$ with a subset of $(\R^{k})^*\otimes\R^{k}$, we have the bundle inclusion: 
$$
    \ufg
    = \P \times_{\SO(k)} \gl(k)
    \subset \P \times_{\SO(k)} (\R^{k})^*         \otimes \P \times_{\SO(k)} \R^{k} 
    = \cE^* \otimes \cE,
$$
and the connection $\nabla^\omega$ on $\ufg$ is the restriction of the tensor product connection on $\cE^* \otimes \cE$.
Alternatively, identifying $\gl(k)\subset (\R^{k})^*\otimes(\R^{k})^*$, we have $\ufg\subset \wedge^2 \cE$ and the connection is induced by restriction accordingly.

\subsubsection{Vertical second fundamental form in $\ufm$} 

Let  $\nabla^\omega$ denote the connection on $\pi^* \ufg$, 
pulled back from the natural connection on the bundle $\ufg$ associated to $\P \to M$.

\begin{lemma}
\label{lemma: 2nd ff}
Acting on diagonal pairs $(X,X)\in \Gamma(TM)\times \Gamma(TM)$, the (vertical) second fundamental form $\nabla^{\cV} d^\cV\sigma$ is expressed in $\ufm$ by:
$$
\cI ((\nabla^{\cV} d^\cV\sigma)(X,X)) = (\nabla^\omega (\sigma^*f ))(X,X)
$$
\begin{proof}
Combining Theorems 3.4 and 3.5 from \cite{Wood2003}, we have:
$$
\cI(\nabla^{\cV} d^\cV\sigma) = \nabla^c (\sigma^* f ) + \tfrac{1}{2} \sigma^* f ^* \cB ,
$$ 
where $\nabla^c$ is the covariant derivative on $\ufm$ (inherited from the principal $H$-bundle $\P \to N$) and $\cB = [.,.]_{\ufm}$ is the $\ufm$-component of the Lie bracket on $\ufg$, inherited from $\fg$,
since we are in the naturally reductive case.

The canonical connection $\nabla^c$ acts on a section  $\alpha\in\Gamma(\ufm)$ by \cite{Wood2003}*{Prop. 2.7} 
\begin{align}\label{eq: canonical connection}
\begin{split}
    \nabla^c \alpha 
    &= \nabla^{\ufm} \alpha - [f ,\alpha]_{\ufm} 
    = (\nabla^\omega \alpha)_{\ufm} - [f ,\alpha]_{\ufm} \\
    &= \nabla^\omega \alpha - [f ,\alpha], 
\end{split}
\end{align}
where  $\nabla^{\ufm}$ is the  $\ufm$-component of the connection $\nabla^\omega$ on $\pi^* \ufg$.
Then 
\begin{align}
\begin{split}
\label{eq1}
    \cI (\nabla^{\cV} d^\cV\sigma) 
    &= (\nabla^\omega (\sigma^*f ))_{\ufm} 
- [\sigma^*f ,\sigma^*f ]_{\ufm} + \tfrac{1}{2} \sigma^* f ^* \cB  \\
    &= (\nabla^\omega (\sigma^*f ))_{\ufm} 
- \tfrac{1}{2}[\sigma^*f ,\sigma^*f ]_{\ufm} \\
    &= \nabla^\omega (\sigma^*f )
- [\sigma^*f ,\sigma^*f ] + \tfrac{1}{2} \sigma^* f ^* \cB ,    
\end{split}
\end{align}
using \cite{Wood2003}*{(2.6)}.
Taking diagonal terms  in \eqref{eq1}:
$$
\cI ((\nabla^{\cV} d^\cV\sigma)(X,X)) = (\nabla^\omega (\sigma^*f ))(X,X)
$$
since all the other terms are skew-symmetric (this is indeed in $\ufm$).
\end{proof}
\end{lemma} 

Finally, taking $\tr_g$ of the expression from Lemma \ref{lemma: 2nd ff} for
$
(\nabla^\omega (\sigma^*f ))(X,X),
$
we obtain the representation in $\ufm$ of the vertical tension field $\tau^\cV(\sigma)$.


\subsection{The Dirichlet energy of a section} \label{Section 1.4}
Suppose in addition that $G$ is semi-simple, so it admits a bi-invariant metric $\eta$, which naturally descends to each homogeneous fibre of $N\simeq \P \times_{G} G/H$, and let $\nabla^\eta$ be its Levi-Civita connection. Using the metrics $(M,g)$ and $(N,\eta)$, there is an induced metric $\left\langle \cdot,\cdot \right\rangle$ on $T^*M \otimes \sigma^*TN$, compatible with the orthogonal splitting $TN=\cV\oplus\cH$. 
This setup admits a (total) Dirichlet action on sections of $N$:
$$
\bar E(\sigma):=\frac{1}{2}\int_M |d\sigma|^2. 
$$

\begin{lemma}
Up to the constant $\rb_M:=\frac{1}{2}(\dim M)(\vol M$), $\bar E(\sigma)$ is completely determined by its vertical component:
$$
\bar E(\sigma)= E(\sigma) + \rb_M
\qwithq
E(\sigma):= \frac{1}{2}\int_M |d^\cV\sigma|^2. 
$$
\end{lemma}
\begin{proof}
    Since $\pi$ is a Riemannian submersion, for a vector field $X$ on $M$, the norm $|d\sigma (X)|^2$ splits into $|d^\cV \sigma (X)|^2 + |d^\cH \sigma (X)|^2$ and, as the horizontal part of the metric on $N$ is the pull-back by $\pi$ of the metric $g$ on $M$, the second term is 
    \begin{align*}
    |d^\cH \sigma (X)|^2 
    &= g (d\pi \circ d \sigma (X),d\pi \circ d \sigma (X)) 
    = g(X,X).
    \qedhere
    \end{align*}
\end{proof}

Seeing a section $\sigma\in\Gamma(\pi)$ as a map from $M$ to $N$, we define its \textbf{tension field} to be $\tau(\sigma)=\tr_g\nabla d\sigma$, whereas its vertical tension field, now as a section, is $\tau^\cV(\sigma):= \tr_g\nabla^\cV d^\cV\sigma$, denoting by $\nabla^{\cV}$ the pull-back by $\sigma$ of the vertical part of the Levi-Civita connection of $(N,\eta)$. 
Moreover, the vertical part of the tension field is precisely the vertical tension field (this is a non-trivial statement, see \cite{Wood2003}), so sections of $\pi : N \to M$ with $\tau^{\cV} (\sigma) =0$ are exactly the critical points of the functional $E(\sigma)$ (and indeed of $\bar E(\sigma)$) for variations through sections of $\pi$. We have the first variation formula:
\begin{proposition}
    $\Crit_{\Gamma(\pi)}(E)=\{ \tau^\cV(\sigma):= 0\}$.
\end{proposition}
\begin{proof}\quad
    Let $\sF:M\times \;]-\epsilon,\epsilon[ \to N$ be a local variation of $\sigma:=\sF(\cdot,0) $ as a section, with (vertical) tangent field 
$$
    V(x):=\left.\partial_t\right\vert_{t=0}\sF(x,t)\in \cV_{\sigma(x)}\subset T_{\sigma(x)}N\simeq (\sigma^*TN)_{x}
$$
    given therefore by $\sF(x,t)=\exp_{\sigma(x)}{tV(x)}$ for small enough $\epsilon$. Then
$$
    \left.\partial_t\right\vert_{t=0} E(\sF(t))
    = \frac{1}{2}\int_M \left.\partial_t\right\vert_{t=0} |d^\cV \sF(x,t)|^2 dx
    = \int_M \left\langle \nabla V,d^\cV\sigma \right\rangle
$$
    and integrating the divergence of the vector field $X:=\left\langle V,d^\cV\sigma \right\rangle_\eta$ yields the claim, since $\cH=\ker\nabla^\omega$ is the orthogonal complement of $\cV$ in $TN$.
\end{proof}

This harmonicity condition is elliptic, in the following sense:

\begin{proposition}
    The vertical tension field $\tau^\cV:\Gamma(\pi)\to \Gamma(\cV)$ is described locally by an elliptic system.
\end{proposition}
\begin{proof}
    The homogeneous bundle $\pi : N \to M$ is equipped with a Sasaki-type metric $g = g^\cH + g^\cV$ with $g^\cH$ the pull-back of the metric on $M$. 
    Let us denote its typical fibre by $F\simeq G/H$. Let  $(U,x^i)$ be a coordinate neighbourhood of $x\in M$.
    
    Given a section $\sigma\in\Gamma(\pi)$, let  $V\subset N$ be an open neighbourhood of $\sigma(x)$. Since $N$ is a fibre bundle, there exist $V_1 \subset M$ in the base manifold and $V_2 \subset F$ in the fibre such that $V= V_1 \times V_2$.
    We take local coordinates $(x^k)$ on $V_1 \subset M$ and $(v^\alpha)$ on $V_2 \subset F$. The horizontal and vertical lifts of the vector fields $(\tfrac{\partial}{\partial x^k})$ and $(\tfrac{\partial}{\partial v^\alpha})$ give a local frame on $V$. Then 
    $$ 
    d^\cV \sigma (\tfrac{\partial}{\partial x^j}) = \sum_{\alpha} \tfrac{\partial\sigma^{\alpha}}{\partial x^j} (\tfrac{\partial}{\partial v^\alpha})^\cV \circ\sigma,
    $$
    and 
\begin{align*}
    (\nabla^\cV d^\cV \sigma)(\tfrac{\partial}{\partial x^i},\tfrac{\partial}{\partial x^j}) 
    &=\nabla^{\cV}_{\tfrac{\partial}{\partial x^i}} (d^\cV\sigma)(\tfrac{\partial}{\partial x^j})
    - d^\cV \sigma(\nabla^{M}_{\tfrac{\partial}{\partial x^i}}\tfrac{\partial}{\partial x^j}) \\
    &=\nabla^{\cV}_{\tfrac{\partial}{\partial x^i}} (\sum_{\alpha} \tfrac{\partial\sigma^{\alpha}}{\partial x^j} (\tfrac{\partial}{\partial v^\alpha})^\cV \circ\sigma)
    - d^\cV \sigma(\Gamma^{k}_{ij}\tfrac{\partial}{\partial x^k}) \\
    &= \sum_{\alpha} \tfrac{\partial^2 \sigma^{\alpha}}{\partial x^i\partial x^j}(\tfrac{\partial}{\partial v^\alpha})^\cV \circ\sigma 
    +\sum_{\alpha} \tfrac{\partial\sigma^{\alpha}}{\partial x^j} \nabla^{\cV}_{\tfrac{\partial}{\partial x^i}} ((\tfrac{\partial}{\partial v^\alpha})^\cV \circ\sigma) - \Gamma^{k}_{ij} d^\cV \sigma(\tfrac{\partial}{\partial x^k}) 
\end{align*}\begin{align*}
    &= \sum_{\alpha} \tfrac{\partial^2 \sigma^{\alpha}}{\partial x^i\partial x^j} (\tfrac{\partial}{\partial v^\alpha})^\cV \circ\sigma 
    + \sum_{\alpha} \tfrac{\partial\sigma^{\alpha}}{\partial x^j} (\nabla^{\cV}_{d\sigma(\tfrac{\partial}{\partial x^i})} (\tfrac{\partial}{\partial v^\alpha})^\cV \circ\sigma)
    - \Gamma^{k}_{ij} \tfrac{\partial\sigma^{\gamma}}{\partial x^k} (\tfrac{\partial}{\partial v^\gamma})^\cV \circ\sigma) \\
    &= \sum_{\alpha} \tfrac{\partial^2 \sigma^{\alpha}}{\partial x^i\partial x^j} (\tfrac{\partial}{\partial v^\alpha})^\cV \circ\sigma 
    + \sum_{\alpha} \tfrac{\partial\sigma^{\alpha}}{\partial x^j} \tfrac{\partial\sigma^{k}}{\partial x^i}
    (\nabla^{\cV}_{(\tfrac{\partial}{\partial x^k})^\cH} (\tfrac{\partial}{\partial v^\alpha})^\cV )\circ\sigma
    \\&\quad+ \sum_{\alpha} \tfrac{\partial\sigma^{\alpha}}{\partial x^j} \tfrac{\partial\sigma^{\beta}}{\partial x^i}
    (\nabla^{\cV}_{(\tfrac{\partial}{\partial v^\beta})^\cV}(\tfrac{\partial}{\partial v^\alpha})^\cV) \circ\sigma
    - \Gamma^{k}_{ij} \tfrac{\partial\sigma^{\gamma}}{\partial x^k} (\tfrac{\partial}{\partial v^\gamma})^\cV \circ\sigma \\
    &= \sum_{\alpha} \tfrac{\partial^2 \sigma^{\alpha}}{\partial x^i\partial x^j}(\tfrac{\partial}{\partial v^\alpha})^\cV \circ\sigma 
    + \sum_{\alpha} \tfrac{\partial\sigma^{\alpha}}{\partial x^j} \tfrac{\partial\sigma^{k}}{\partial x^i}
    G^{\gamma}_{\alpha k} (\tfrac{\partial}{\partial v^\gamma})^\cV \circ\sigma\\
    &\quad+ \sum_{\alpha} \tfrac{\partial\sigma^{\alpha}}{\partial x^j} \tfrac{\partial\sigma^{\beta}}{\partial x^i}
    G^{\gamma}_{\alpha \beta} (\tfrac{\partial}{\partial v^\gamma})^\cV \circ\sigma
    - \Gamma^{k}_{ij} \tfrac{\partial\sigma^{\gamma}}{\partial x^k} (\tfrac{\partial}{\partial v^\gamma})^\cV    
     \circ\sigma)\\
     &= \sum_{\alpha} \Big(\tfrac{\partial^2 \sigma^{\alpha}}{\partial x^i\partial x^j} +
     \tfrac{\partial\sigma^{\alpha}}{\partial x^j} \tfrac{\partial\sigma^{k}}{\partial x^i}
    G^{\gamma}_{\alpha k} 
    + \tfrac{\partial\sigma^{\alpha}}{\partial x^j} \tfrac{\partial\sigma^{\beta}}{\partial x^i}
    G^{\gamma}_{\alpha \beta}  - \Gamma^{k}_{ij} \tfrac{\partial\sigma^{\gamma}}{\partial x^k} \Big) (\tfrac{\partial}{\partial v^\gamma})^\cV  \circ\sigma, 
\end{align*}
    where
\begin{align*}
    G^{\gamma}_{\alpha \beta}
    &=\sum_{\delta} g^{\delta\alpha} g\left( \nabla_{(\tfrac{\partial}{\partial v^\beta})^\cV} (\tfrac{\partial}{\partial v^\gamma})^\cV , (\tfrac{\partial}{\partial v^\delta})^\cV \right)\\
    G^{\gamma}_{k \beta}
    &=\sum_{\delta} g^{\delta\alpha} g\left( \nabla_{(\tfrac{\partial}{\partial x^k})^\cV} (\tfrac{\partial}{\partial v^\gamma})^\cV , (\tfrac{\partial}{\partial v^\delta})^\cV \right).
\end{align*}
    This concludes the proof. Notice that $G^{\gamma}_{\alpha \beta}$ and $G^{\gamma}_{k \beta}$ do not commute in their lower indices.
\end{proof}
To obtain the horizontal part of $\tau(\sigma)$, we take the second covariant derivative of $\pi\circ\sigma = \id$ and use standard properties of Riemannian submersions with totally geodesic fibres:
\begin{align*}
    -2 g\left( d\pi((\nabla d\sigma)(X,Y)), Z\right) 
    &= 2 g\left( (\nabla d\pi)(d^\cV\sigma(X), d^\cH\sigma (Y) + (\nabla d\pi)(d^\cH\sigma(X), d^\cV\sigma (Y) , Z\right) \\
    &= \lan (\sigma^\ast f )(X) , f  [d^\cH\sigma(Y),d^\cH\sigma(Z)] \ran +\lan (\sigma^\ast f )(Y) , f  [d^\cH\sigma(X),d^\cH\sigma(Z)] \ran.
\end{align*}
On the other hand, the structure equation for $\omega$, projected on $N$, implies that 
$$ -f  [H,K] = F (H,K), \quad \forall H,K \in \cH,$$
where $F$ is the $\ufm$-valued 2-form on $N$ obtained as the projection of the $\fm$-component of the curvature form of $p: \P \to M$. Therefore 
\begin{align*}
    2 g\left( d\pi((\nabla d\sigma)(X,Y)), Z\right) &= \lan (\sigma^\ast f )(X) , F  (\sigma(Y),\sigma(Z)) \ran +
    \lan (\sigma^\ast f )(Y) , F  (\sigma(X),\sigma(Z)) \ran .
\end{align*}
The vanishing of the horizontal part of the tension field of $\sigma$ will therefore be given by the condition
$$
\sum_{i=1}^{7} \lan (\sigma^\ast f )(e_i) , (\sigma^\ast F ) (e_i , X) \ran = 0, \quad \forall X \in TM.
$$

\subsection{Equivariant lifts from sections to maps}
\label{sec: sigma <-> s}

There is a natural $1-1$ correspondence among sections of $\pi : N \to M$ and $G$-equivariant maps from $\P$ to $G/H$, due to the isomorphism between $\P/H$ and the associated bundle $\P \times_{G} G/H$, given by
\begin{align}
\label{eq: avb correspondence}
[(z,gH)]_G \in P \times_{G} G/H \mapsto [z.g]_H \in P/H .
\end{align}

\begin{lemma}
\label{lem: mu is an isometry}
    For each $z\in \P$, the map (cf. \S\ref{sec: canonical geom on hfb})
\begin{align}
\label{eq: mu 1-1 correspondence}
\begin{split}
    \mu_z : G/H &\to N_{p(z)} \subset \P\times_{G} G/H \\
    a &\mapsto z\bullet a
\end{split}
\end{align}
    defines an isometry of $G/H$ onto the fibre of $\pi$ over $p(z)$, with respect to the bi-invariant metric on $G$.
\end{lemma}
\begin{proof}
    Since this is a crucial fact, we give a little more detail for this statement found in \cite{Wood1997}. 
    Given $z\in P$, the differential of $\mu_z$ at $a\in G/H$ is:
\begin{align}
\begin{split}
    (d\mu_z)_a : T_a(G/H) &\to T_{z\bullet a}N \\
    X_a &\mapsto [X_a]_{\Ad(G)}.
\end{split}
\end{align}    
Now, the metric $\eta$ on $G/H$ is $G$-invariant, by assumption, so it goes over identically to $\Ad(G)$-orbits. 
\end{proof}
Then, to any section $\sigma\in\Gamma(\pi)$, seen as a section of $\P \times_{G} G/H$, we can associate bijectively the G-equivariant map $s : \P \to G/H$ defined by:
\begin{align}
\label{eq: mu 1-1 correspondence maps}
    \sigma (x) = \mu_z(s)
    := z \bullet s(z) \qwithq p(z)=x .
\end{align}
One can easily check that this does not depend on the choice of $z\in\P_x$  and that $s$ is equivariant, i.e. $s(z.g)= g^{-1}.s(z)$.
Conversely, since $\mu_z$ is onto the fibre of $\pi$ over $p(z)$, the inverse association is
$$
s(z) = \mu_z^{-1}(\sigma(x)) \qwithq p(z)=x.
$$
In summary, the maps 
\begin{itemize}
    \item[\eqref{eq: Xi correspondence}]
    $\xi_\sigma=\Xi\circ \sigma$
    \item[\eqref{eq: avb correspondence}]
    $[(z,gH)]_G  \leftrightarrow [z.g]_H $
    \item[\eqref{eq: mu 1-1 correspondence maps}]
   $ \sigma (x) = z \bullet s(z) = [(z,s(z))]_G$
\end{itemize}
enable us to associate to a $H$-structure $\xi_\sigma\in\Gamma(M,\cF)$ a $G$-equivariant  map $s : \P \to G/H$, through a section $\sigma$, interpreted under two guises, as $\sigma : M \to N = \P/H$ and $\sigma : M \to \P\times_{G} G/H$:
$$
\{\xi_\sigma : M \to \cF\} 
\stackrel{\eqref{eq: Xi correspondence}}{\longleftrightarrow}
\{\sigma : M \to \P/H\}
\stackrel{\eqref{eq: avb correspondence}}{\longleftrightarrow}
\{\sigma : M \to \P\times_{G} G/H\}
\stackrel{\eqref{eq: mu 1-1 correspondence maps}}{\longleftrightarrow}
\{s : \P \to G/H, \mbox{ $G$-eq.}\}
$$

\begin{lemma}
\label{lem: various torsions}
    Under the correspondences 
    \eqref{eq: Xi correspondence}, \eqref{eq: avb correspondence} and
    \eqref{eq: mu 1-1 correspondence maps}, between geometric structures $\xi:M\to\cF$, sections $\sigma:M\to N$, and $G$-equivariant maps $s:P\to G/H$,
    $$
    ds=0
    \quad\Rightarrow\quad
    d^\cV\sigma=0
    \quad\Rightarrow\quad
    \nabla\xi_\sigma=0.
    $$
\end{lemma}
\begin{proof}
    Since $\mu$ is an isometry, the first implication follows directly from the formula in \cite{Wood1997}*{Lemma~1}: 
    $$ 
    d\mu_z \circ ds(Z) 
    = d^\cV \sigma (X),
    $$
    for a vector $X\in TM$ and any  $Z$ in the horizontal distribution of $p: P\to M$, such that $p_*Z=X$. 

    For the second implication, consider the universal section $\Xi : N \to \pi^*(\cF)$ and the $H$-structure $\xi_\sigma =\Xi\circ \sigma : M \to \cF$. Since $\sigma^{-1}(\pi^*(\cF)) = \cF$, the expression  \eqref{eq: diff} of the covariant derivative of a universal structure in terms of the homogeneous connection form $f$ yields: 
\begin{align}
\begin{split}
\label{rel-lemma6}
    \nabla_{X}^{\cF} \xi_\sigma 
    &=  \nabla_{X}^{\sigma^{-1}(\pi^*(\cF))} (\Xi\circ \sigma )
     = (\nabla_{d\sigma(X)}^{\pi^*(\cF)} \Xi)\circ \sigma 
     = (f(d\sigma(X)) .\Xi)\circ \sigma, \\
    &= (f(d^\cV\sigma(X)) .\xi_\sigma, 
\end{split}
\end{align}
    since the horizontal distribution of $\pi$ lies in $\ker f$.
\end{proof}

\begin{remark}
If $\fm$ is an irreducible representation of $H$ then, by Schur's lemma, there exists an isomorphism mapping $(f(d^\cV\sigma(X)) .\xi_\sigma$ to $d^\cV\sigma(X)$, so the last implication of Lemma~\ref{lem: various torsions} becomes an equivalence.
Moreover, if, in the irreducible decomposition of the $2$-symmetric powers of $\fm$ under $H$-action, the trivial representation appears only once,
then the norms (in their respective spaces) of $\nabla_{X}^{\cF} \xi_\sigma$ and $d^\cV\sigma(X)$ are proportional, by a multiplicative constant.
\end{remark}

Comparing the natural harmonicity theories of a section and its corresponding map $s=\mu^{-1}(\sigma)$;  one easily obtains:
\begin{lemma}
\label{lemma: mu relates harmonicities}
    In terms of the constants $\ra_P:=\vol G$ and $\rb_P:=\frac{1}{2}\dim G\vol P$, for every $z\in\P$ and $\sigma=\mu_z(s)$,
    $$
    E(s):= \frac{1}{2}\int_\P |ds|^2
    =\ra_P .E(\sigma) + \rb_P.
    $$
\end{lemma}
\begin{proof}
    Let $\{X_i\}$ be an orthonormal frame of $M$ and $\{\tilde{X}_i,a_j^*\}$ an orthonormal frame of $P$ such that each $\tilde{X}_i$ is the horizontal lift of $X_i$ and $a_j^*$ the (vertical) fundamental vector field generated by $a_j\in \fg$.
    Then, in terms of the fibre metric $|\cdot|=|\cdot|_\eta$,
    \begin{align*}
        E(s) 
        &= \frac{1}{2}\int_P |ds|^2 
        = \frac{1}{2}\int_P |ds(\tilde{X}_i)|^2 + |ds(a_j^*)|^2 \\
        &= \frac{1}{2}\int_P |d^\cV \sigma(X_i)|^2 + |ds(a_j^*)|^2, \quad \text{by \cite{Wood1997}*{Lemma 1}}\\
        &= \frac{1}{2}\vol (G) \int_M |d^\cV \sigma(X_i)|^2 + \frac{1}{2}\int_P |ds(a_j^*)|^2, \quad \text{by the co-area formula}
    \end{align*}
    But \cite{Wood1997}*{Proposition 1} shows that $d^\cV \sigma(dp(Z)) = d\mu_q ( ds(Z) +\omega(Z)_{s(p)})$, for any $Z\in TP$, so for the vertical vector fields $a_j^*$, we have
    $ \sum_j |ds(a_j^*)|^2 
        = \sum_j |a_j|^2 
        = \dim G$.
\end{proof}

\begin{remark}
We learn in \cite{Wood1997}*{Theorem 1} that $\sigma$ is harmonic (as a section) if, and only if, $s$ is horizontally harmonic, i.e. the trace, over the horizontal distribution, of its second fundamental form vanishes. This corrects a statement which originally appeared in \cite{Wood1990}, by organising it as two distinct sets of conditions: 
if $G$ is unimodular and $G/H$ compact with non-positive Ricci curvature, or if $G/H$ is a normal $G$-homogeneous manifold (cf. Definition~\ref{def: normal subgroup}) and the metric on $\P$ is constructed from any compatible metric on $G$, then $\sigma$ is a harmonic section if and only if $s$ is a harmonic map.
\end{remark}


\section{The harmonic section flow of geometric structures}
\label{sec: HSF}

In the setup of Assumptions \eqref{eq: conditions of HSF}, suppose $G/H$ is a normal $G$-homogeneous manifold,  and the metric on $\P$ is constructed from any compatible metric on $G$, cf. \cite[Def. 7.8 \& 7.86]{Besse2008}. Given a smooth initial  section $\sigma_0$  of the  fibre bundle $\pi:N\to M$, consider the natural  \emph{harmonic section flow} \eqref{eq: HSF}:
\begin{equation*}
    \begin{cases}
    \displaystyle  \partial_t\sigma_t &= \tau^\cV (\sigma_t)\\
    \sigma_t|_{t=0} &= \sigma_0
    \end{cases},
\quad\text{on}\quad 
    M_{\rT}:=M\times \left[0,\rT\right[. 
\end{equation*}
The main purpose of this paper is to initiate an analytic theory for this flow.

\subsection{Motivation for the vertical flow}
A first simple fact behind the formulation of \eqref{eq: HSF} is the property that the usual harmonic map heat flow starting from a section $\sigma_0 : M\to N$ remains among sections for as long as it exists:
\begin{proposition}
\label{prop HMF for sections}
    Let $\sigma_0 :M \to N$ be a section of $\pi$ and assume that for all $x\in M$, $N_x = \pi^{-1}(x)$ is a totally geodesic submanifold of $N$.
    Let $I$ be an interval and $\{u_t\}_{t\in I}\subset C^\infty(M,N)$ a solution of \textcolor{blue}{the harmonic map heat flow}
    $$
    \partial_t u_t = \tau (u_t),
    \quad  u_0=\sigma_0.
    $$
    Then $u_t$ is a section of $\pi$, for all $t\in I$.
\end{proposition}
\begin{proof}
Use Nash's isometric embedding theorem to see $N$ as a submanifold of a Euclidean space $\R^m$.
Then the tension field of a map $u : M \to N \subset \R^m$ is expressed in terms of the the second fundamental form of $N$ by $\tau (u) = \Delta u - \rII_{N}(du,du)$. Since the base manifold $M$ and the fibres $N_x$ are compact, for all $x\in M$, there exists a uniform $\epsilon >0$ such that every $N_x$ admits a tubular $\epsilon$-neighbourhood $N^\epsilon_x$, together  with a submersion
$\pr_x :  N^\epsilon_x \to N_x $. Each $\pr_x (y)$ can be assumed to be the orthogonal projection of $y\in N^\epsilon_x$ onto $N_x$, and the collection of such difference vectors  in $\R^m$ defines a function
\begin{align*}
    \rho_x :  N^\epsilon_x 
    &\to \ker(\pr_x)\subset\R^m \\
    y &\mapsto \rho_x (y) = y - \pr_x (y).
\end{align*}

For small enough $t>0$, define the squared-distance $\ell_t(x) = |\rho_x (u(x,t))|^2$  on $M\times I$. Clearly $\ell_0\equiv0$, since $u_0$ is a section.
Then 
\begin{align*}
    \partial_t \ell  
    &= \partial_t\langle \rho_x (u) , \rho_x (u) \rangle 
     = 2 \langle d\rho_x (\partial_t u  ), \rho_x (u) \rangle \\
    &= 2 \langle d\rho_x (\Delta u - \rII_{N}(du,du) ), \rho_x (u) \rangle
\end{align*}
On the other hand,
\begin{align*}
    \Delta \ell 
    &= \Delta \langle \rho_x (u) , \rho_x (u) \rangle \\
    &= 2  \langle \Delta \rho_x (u) , \rho_x (u) \rangle + 2 |\nabla \rho_x (u)|^2 
\end{align*}
Now, $ \rho_x  = \id - \pr_x$ implies $\nabla d \rho_x = - \nabla d \pr_x$, and $\tr \nabla d \pr_x$ is the second fundamental form of $N_x \subset \R^m$, so
\begin{align*}
    \Delta \rho_x (u) 
    &= d\rho_x (\Delta u) + (\tr \nabla d \rho_x)(du,du) \\
    &= d\rho_x (\Delta u) - (\rII_{N_x})(du,du) ,
\end{align*}
In conclusion,
\begin{align*}
    \Delta \ell 
    &=2  \langle d\rho_x (\Delta u) , \rho_x (u) \rangle - 2  \langle (\rII_{N_x})(du,du) , \rho_x (u) \rangle  + 2 |\nabla \rho_x (u)|^2. 
\end{align*}
Furthermore, $N_x$ is totally geodesic in $N$, so $\rII_{N_x} = \rII_{N}$ and, since $d\rho_x$ is the identity on normal vectors,
$$
\partial_t \ell  - \Delta \ell = - 2 |\nabla \rho_x (u)|^2 .
$$
We conclude that $\ell_T \equiv0$ by the divergence theorem:
\begin{align*}
    &\partial_t 
    \int_{M} \ell_t  
    =  \int_{M} \partial_t \ell_t 
    = -2  \int_{M} |\nabla \rho_x (u(x,t))|^2  \leq 0,\\
    \Rightarrow\quad
    & 0\leq \int_{M} \ell_T  \leq \int_{M} \ell_0  =0,
    \quad
    \forall T>0.   
    \qedhere
\end{align*}
\end{proof}
As a result, the harmonic map heat flow will evolve an initial section through sections, hence $\partial_t u_t$ will be vertical. We may therefore  relax that flow and  work just with the vertical part, which motivates our formulation of \eqref{eq: HSF}.

\subsection{A priori estimates along the harmonic section flow}
We consider a flow of geometric structures corresponding to sections $\sigma\in\Gamma(\pi)$ as in \eqref{eq: HSF}. A number of preliminary 
properties can be derived from the outset, reflecting the close resemblance between \eqref{eq: HSF} and the classical harmonic map flow, cf. \cite{Nishikawa2002}.
We adopt the following conventions for norm estimates in this time-dependent context:
$$
\sigma\in L^p(M_\rT) 
\Leftrightarrow
\sigma(\cdot,t)\in L^p(M), \forall t\in \left[0,\rT\right[.
$$
For notational convenience, let us introduce the \emph{vertical torsion} 
$$
T:=d^\cV\sigma.
$$ 
We define respectively the \emph{Dirichlet action} and the \emph{kinetic energy}, by
$$
E(t):=\frac{1}{2}\int_M |T(t)|^2 _{g,\eta}
\qandq
K(t):=\frac{1}{2}\int_M |\tau^{\cV}(\sigma(t))|^2_{g,\eta}
$$
as well as their pointwise densities
$$
\epsilon(t):=\frac{1}{2}|T(t)|^2_{g,\eta}
\qandq
\kappa(t):=\frac{1}{2}|\tau^{\cV}(\sigma(t))|^2_{g,\eta}.
$$
Since \eqref{eq: HSF} is the gradient flow of $E(t)$, we should expect the energy to be non-increasing, and indeed a direct computation gives:
\begin{lemma}
\label{lemma: E(t) unif bounded}
$E'(t)=-2K(t)\leq0$, hence $E(t)\leq E_0:=E(0)$ and therefore $T=d^\cV\sigma$ is uniformly bounded in $L^2$.
\end{lemma}

Along the flow \eqref{eq: HSF}, the densities $\epsilon$ and $\kappa$ are subsolutions of non-homogeneous heat equations:
\begin{proposition}
\label{prop: heat eqs of e and k}
    Denote the heat operator on $M_\rT$ by
    $$
    \sH:=\partial_t-\Delta.
    $$
    There exist constants $C_1, C_2>0$, depending only on $(M,g)$, such that any solution of \eqref{eq: HSF} satisfies the following inequalities:
    \begin{align}
        \tag{i}
        \sH(\epsilon)=\left(\partial_t-\Delta\right)\epsilon
        &\leq C_1 (\epsilon^2+\epsilon + 1)     
        - |\nabla^\omega T|_\eta^2,\\
        \tag{ii}
        \sH(\kappa)=\left(\partial_t-\Delta\right)\kappa
        &\leq C_2\epsilon\kappa 
        - |\nabla^\omega\tau^\cV(\sigma)|^2.
    \end{align}
\begin{proof} 
    Following \cite{Nishikawa2002}*{Thm 4.2}, we apply Schwartz's lemma for $C^\infty$-sections $\sigma\in\Gamma(\pi)$, in normal coordinates $g_{ij}=\delta_{ij}$ on $M$:
    \begin{equation}
    \label{eq: time-derivative of epsilon}
        \partial_t\epsilon
        =\frac{1}{2}\partial_t\left\langle d^\cV\sigma,d^\cV\sigma\right\rangle_{g,\eta}
        =\sum_{j=1}^m \left\langle  
        \nabla_j^\omega \dot\sigma ,  d_j^\cV\sigma 
        \right\rangle_\eta.    
    \end{equation}
    On the other hand, the Laplacian of $\epsilon$ is expanded as follows:
    \begin{align*}
        \Delta\epsilon
        &=\tr_g \nabla\nabla\epsilon 
        =\frac{1}{2}\sum_{i=1}^m \nabla_i\left(\nabla_i\left\langle  d^\cV\sigma,d^\cV\sigma
        \right\rangle_{g,\eta}\right)
        =\sum_{i,j=1}^m \nabla_i\left(\left\langle 
        \nabla_i^\cV  d_j^\cV\sigma  ,  d_j^\cV\sigma 
        \right\rangle_\eta\right)\\
        &= \sum_{i,j=1}^m \left\langle 
        \nabla_i^\cV\nabla_i^\cV  d_j^\cV\sigma  ,  d_j^\cV\sigma 
        \right\rangle_\eta 
        + |\nabla^\cV d^\cV\sigma|_\eta^2,
    \end{align*}
    Adopting the curvature convention
    \begin{equation}
    \label{eq: curvature convention}
        R(X,Y) = [\nabla_X , \nabla_Y] - \nabla_{[X,Y]},
    \end{equation}
    the Ricci identity for the product connection $\nabla:=\nabla^g\otimes \sigma^*\nabla$ on $TM\otimes \sigma^*TN$ reads \cite{Nishikawa2002}*{pp.125,161}:
    \begin{equation}
        \nabla_i^\cV\nabla_i^\cV  d_j^\cV\sigma 
        =\nabla_j^\cV\nabla_i^\cV  d_i^\cV\sigma 
        +(R_M)_{ij} (d_i^\cV\sigma) 
        -\left((R_N)(d_i^\cV \sigma,d_j^\cV \sigma)\right)
        d_i^\cV \sigma,
    \end{equation}   
    and $\nabla_i^\cV  d_i^\cV\sigma=\tau^\cV(\sigma)=\dot\sigma$ along the flow.
    This yields the first claim, since both $M$ and $N$ are compact, hence have bounded curvature. The second claim is proved in a similar manner [ibid.].
\end{proof}
\end{proposition}

As an immediate consequence, we have a first a priori regularity estimate for both $T$ and $\tau^\cV$:
\begin{proposition}
\label{prop: tauV in L^2_1}
    If a solution to \eqref{eq: HSF} exists on $M_\rT$ and the kinetic energy  $K(t)=\Vert\kappa(t)\Vert_{L^1(M)}$ is bounded uniformly in $t$, then:
    \begin{enumerate}[(i)]
        \item
        $\Vert T\Vert_{L^2_1(M)}^2\lesssim \Vert T\Vert_{L^4(M)}^4.$
        
        \item 
        If $\rT=\infty$, there exist $C>0$ and a sequence $t_n\nearrow^\infty$ such that $\Vert \tau^\cV(t_n)\Vert_{L^2_1(M)}\leq C$.
    \end{enumerate}
\begin{proof}
Integrating over $M$ the inequalities from Proposition \ref{prop: heat eqs of e and k}, and applying Lemma \ref{lemma: E(t) unif bounded}, we have respectively:
    \begin{enumerate}[(a)]
        \item 
        $\Vert T(t)\Vert_{L^2_1(M)}^2\lesssim \Vert T(t)\Vert_{L^4(M)}^4 +2K(t)+E_0+1$.
        
        \item
        $K'(t)\leq E_0 \sup_M \kappa(t)
        -\Vert \tau^\cV(t)\Vert^2_{L^2_1(M)}$.  
    \end{enumerate}
    Assertion \emph{(i)} is immediate from (a) and our hypothesis on $K(t)$.
    For \emph{(ii)}, let $B:=\sup K(t)$, and consider the `doubling' intervals 
    $$
    I_j:=\left[2^j,2^{j+1}\right].
    $$
    Suppose, by contradiction, that for every $C>0$  there exists $j>0$ such that 
    $$
    \Vert \tau^\cV(t)\Vert_{L^2_1(M)}> C,
    \quad\forall t\in I_j.
    $$
    This has to be the case, for if the inequality failed so much as at a single instant $t_j$ in each $I_j$, the sequence $\{t_j\}$ would satisfy the assertion. Then, for $C$ sufficiently large, inequality (b) above implies $K'|_{I_j}\lesssim - C$, and so:
    $$
    -B\leq \int_{\partial I_j} K = \int_{I_j} K'(t)dt\lesssim -2^jC,
    $$
    which is eventually false, for $j\gg0$.
\end{proof}
\end{proposition}

\subsection{The harmonic map heat flow of $G$-equivariant maps}

Following an insight  formulated in \cite{Wood1990}, we will see that the harmonic section flow \eqref{eq: HSF} is closely related to the natural \emph{harmonic map heat flow} of a given $G$-equivariant map $s_0:P\to G/H$:
\begin{equation}
\label{eq: HMHF}
\tag{HMHF}
    \left\{\begin{array}{rcl}
    \displaystyle \partial_t s_t
        &=& \tau(s_t) := \tr\nabla ds\\
    s_t|_{t=0}
        &=&s_0
    \end{array}\right.,
    \quad\text{on}\quad 
    \P_{\rT}:=\P\times \left[0,\rT\right[, 
\end{equation}
An immediate question in this context is whether this flow preserves $G$-equivariance, or whether it flows some initial $s_0$ merely as a map from $P$ to $G/H$. 
Recall that a family $f:=\{f_t\}  \in C^1 (X\times \R,Y)$ of maps between $G$-manifolds $X$ and $Y$ is \emph{equivariant} if, and only if, each $f _t$ is equivariant, i.e. $g.f  = f $.
The induced action on vector fields is:
$$ (g.V)(x) = dI_g (V(g^{-1}.x)) ,
\quad\forall g\in G,
$$
where $I_g$ is the diffeomorphism corresponding to $g$, which we declare to be an isometry. 

\begin{lemma}[{\cite{Wood1990}*{Lemma 2.1}}] 
\label{lem: G-equiv flow}
    If $P$ and $Q$ are (left) $G$-manifolds, then the action extends naturally to $C^1 (P,Q)$ by 
    $$ 
    (g.s )(x) = g.s (g^{-1} .x), \quad g\in G, x\in P
    $$
    and trivially to families $s :=\{s_t\}\in C^1 (P\times \R,Q)$. Given an equivariant family $s  \in C^1 (P\times \R,Q)$,
    $$ 
    \cL (g.s ) = g.\cL (s ) 
    \quad\forall g \in G,
    $$
    where $\cL (s ) := \tau(s) - \partial_t s$ is the `heat' operator.
\end{lemma}

This lemma implies that the flow with equivariant initial value remains equivariant, in particular, if $s_t$ is a heat flow then $g.s_t$ is also a heat flow. When both $M$ and the structure group $G$ are compact, hence also the total space of the principal $G$-bundle $P\to M$, and assuming that the quotient $G/H$ has non-positive sectional curvature, then one could consider invoking the celebrated Eells-Sampson Theorem \cite{Eells1964} to obtain a limiting equivariant harmonic map $s_\infty:P\to G/H$, which, by uniqueness, is also equivariant. It corresponds therefore also to a harmonic section $\sigma_{\infty} : M \to N$ under the isometry $\mu$ in \eqref{eq: mu 1-1 correspondence}, by Lemma \ref{lem: various torsions}. That occurs indeed if, and only if, $\fg=\Lie(G)$ is a \emph{$NC$ algebra} \cite{Azencott1976}.

Unfortunately, that is seldom the case among the relevant groups in the context of geometric structures. We are, quite often, in the exactly opposite case, eg. when $G/H$ is assumed normal, as in our assumptions \eqref{eq: conditions of HSF}, and therefore has nonnegative sectional curvature. We must therefore address the flow \eqref{eq: HSF} in full generality.

\subsection{Correspondence between \eqref{eq: HSF} and \eqref{eq: HMHF}}\label{section correspondence}

In the absence of curvature assumptions, one can still prove several analytic properties of the harmonic section flow, using the fact that the target space is a homogeneous fibre bundle, i.e. a `bundle of homogeneous spaces'.
The crucial metric compatibility hypothesis in our assumptions \eqref{eq: conditions of HSF} is the normality of the induced metric on the fibres of $N=P/H$:

\begin{definition}[{\cite{Besse2008}*{Def. 7.86}}]
\label{def: normal subgroup}
    A $G$-homogeneous Riemannian manifold $(G/H,\eta)$ is called \emph{normal} if there exists an $\Ad(G)$-invariant scalar product $\bar{\eta}$ on $\fg$ such that, if $\fm\subset \fg$ is the $\bar{\eta}$-orthogonal complement of $\fh$, then $\eta$ coincides with the restriction $\bar{\eta}|_\fm$.
\end{definition}

Normal metrics are naturally reductive \cite{Kobayashi1969a}*{Cor. 3.6} and have non-negative sectional curvature. In particular, given from the outset a left-invariant metric on $G$, one can take directly $\eta$ as its restriction to $\fm:=\fh^\perp$.
The following result effectively proves Theorem \ref{thm: uniqueness and s-t existence}:

\begin{proposition}
\label{prop: HMHF}
    Under the assumptions \eqref{eq: conditions of HSF},  a  family $\{\sigma_t\}\subset \Gamma(\pi)$ is a solution of the harmonic section flow \eqref{eq: HSF}, for some $\rT>0$, if, and only if, the corresponding family of $G$-equivariant lifts $s_t=\mu^{-1}(\sigma_t):\P\to G/H$   in \eqref{eq: mu 1-1 correspondence} is a solution of  the harmonic map heat flow  \eqref{eq: HMHF} with initial condition $s_0:=\mu^{-1}(\sigma_0)$.

    In particular, any initial condition $\sigma_0 \in\Gamma(\pi)$ determines a (short) time $\rT>0$ and a unique solution $\{\sigma_t\}\subset \Gamma(\pi)$ of \eqref{eq: HSF} on $M_\rT$.
\end{proposition}
\begin{proof}
    Let $s_0 : \P \to G/H$ be the $G$-equivariant lift of $\sigma_0$. Then there exists a (short) time $\rT >0$ such that, for all $t\in [0,\rT)$, $\{s_t : \P \to G/H\}$ is a solution to \eqref{eq: HMHF}, cf. \cite{Nishikawa2002}*{Theorem 4.10}.  
    By equivariance from Lemma~\ref{lem: G-equiv flow} and uniqueness of the short-time harmonic map heat flow for smooth initial data \cite{LinWang2010}, $s_t$ and $gs_t$ are harmonic flows with the same initial value, hence equal. Moreover, the family $\{s_t\}$ is equivariant, and it gives rise to sections $\sigma_t : M \to G/H$, for $t\in [0,T[$.

    We now pull-back the $G$-equivariant heat flow equation $\partial_t s_t = \tau(s_t)$  to the flow satisfied by $\sigma_t$. 
    For $z\in \P$, the isometry $\mu_z:G/H\to \P$ relates the corresponding tension fields by
    $$
    \tau^\cV (\sigma_t)=
    d\mu_z \Big(\partial_t s_t\Big).
    $$
    To further express the vertical torsion  in terms of $\partial_t \sigma_t$, we consider the maps $\Sigma (x,t) = \sigma_{t}(x)$ and $S(x,t) = s_{t}(x)$, defined on the Cartesian product $M_\rT $ of $M$ by a real interval,
    and denote by $\partial_t\in T(M_\rT )$ the vector field along the `time' direction. Then $ \partial_t s_t = dS(\partial_t)$ and $\partial_t \sigma_t = d\Sigma(\partial_t)$.
    \begin{multicols}{2}
    Extending trivially the $H$-action along $I$, we have a natural homogeneous bundle construction on $M_\rT $ (with time-dependent versions of the maps on page~\pageref{figure 1}), since $(\P \times \R^*)/H = \P/H \times \R^* = N \times \R^*$. Define $\Sigma: M_\rT  \to N\times \R^*$ by 
    $$
    (x,t) \mapsto \Sigma(x,t) 
    = (\sigma_{t}(x), r)
    $$
    to be a variation of $\sigma$ 
    through sections of the homogeneous bundle $\pi : N \to M$.
    \columnbreak
    $$\xymatrix{
        \llap{$(z,r)\in$ } \P \times \R^*
	    \ar[dr]^{\rmq} 
        \ar[dd]_{\rp}    \\ 
        &\llap{$(y,r)\in$ } N \times \R^*
        \ar[dl]_{\uppi} \\ 
        \llap{$(x,t)\in$ }  M_\rT  \ar@/_1pc/@{-->}[ur]!D_{\Sigma}& &
    }$$
    \end{multicols}
    Therefore, identifying $\uppi$ and $\pi$,
    $d(\pi\circ \Sigma)(x,t)) = x, \, \forall t\in [0,\rT[$, hence the differential of $\pi\circ \Sigma$ with respect to the variable $t$ vanishes, that is $d\pi (d\Sigma(\partial_t)) =0$. So $d\Sigma(\partial_t)\in \cV$ and, since $\rmq_{*}$ is the identity on the second factor of $\P\times \R$ (in particular, vectors tangent to the second factor of $\P\times \R$ are horizontal for $\rp$), the formula 
    $$
    d\mu_{(z,r)} (DS(Z)) = d^\cV \Sigma(\rmq_{*} Z) , \, \forall Z\in T\P\times\R ,
    $$
    implies, for $Z=\partial_t$, that
    $$
    d\mu_{(z,r)} (DS(\partial_t)) 
    = d\mu_{(z,r)} (dS(\partial_t)) 
      = d\mu_{z} \Big(\partial_t s_t\Big)
    $$
    since the map $(z,r)\bullet$ is trivial on the second factor, while $d^\cV \Sigma(\rmq_{*} \partial_t) = d^\cV \Sigma( \partial_t) = d \Sigma( \partial_t)  = \partial_t \sigma_t.$
    The pull-back of the heat flow equation for $s_t$ is therefore a solution of \eqref{eq: HSF}.
\end{proof}

In summary, the relation 
$\mu_z(s_t(z))=(\sigma_t\circ p)(z)$ from \eqref{eq: mu 1-1 correspondence} gives a  1-1 correspondence  
\begin{equation}
\label{eq: (HSF) <-> (HMHF)}
    \{\sigma_t:M\to N \mid \text{solution of } \eqref{eq: HSF} \}
    \quad\leftrightarrow\quad 
    \{s_t:P\to G/H \mid \text{solution of } \eqref{eq: HMHF} \},
\end{equation}
between sections of $N$ and $G$-equivariant maps.
Since  $\mu$ is an isometry [Lemma~\ref{lem: mu is an isometry}], the flow \eqref{eq: HSF} inherits some well-known properties from  \eqref{eq: HMHF}, as stated in Theorem \ref{thm: Tmax and doubling time}. Indeed, the elementary finite-time blow-up result Theorem \ref{thm: Tmax and doubling time}--(i) is a direct consequence of \cite{LinWang2010}*{Theorem 5.2.1} (quoted there from \cite{Eells1964}*{pp.154-155}), applied to a solution $\{s_t\}$ of \eqref{eq: HMHF}; and Theorem \ref{thm: Tmax and doubling time}-(ii) follows immediately from the Weitzenb\"ock formula in Proposition \ref{prop: heat eqs of e and k}--(i) and the maximum principle, see eg. the proof of \cite{Dwivedi2019}*{Proposition 3.2}.

\begin{remark}
Huang and Wang \cite{Huang2016} show smooth regularity and uniqueness of the heat flow of harmonic maps into a general closed Riemannian manifold, with initial data in $L^2_1$ for Serrin's ($p,q$)-solutions, i.e. under a small parabolic Morrey norm condition. This, combined with the correspondence between sections and $G$-equivariant maps of section~\ref{sec: sigma <-> s}, could pave the way for a study of the harmonic evolution of weak geometric structures.
\end{remark}

\subsection{Conditions for long-time existence}
\label{sec: long-time existence}

\subsubsection{Regularity of solutions to a parabolic flow}
Our next result exploits the regularising nature of parabolic flows, by means of the Harnack-Moser estimate:

\begin{lemma}[{\cite{Lin2008}*{\S 5.3.4}}]
\label{lemma: parabolic cylinder}
    Let $D\subset\R^{n+1}$ be a domain containing the parabolic cylinder 
    $$
    \Sigma_R(x_0,t_0):=
	    \left\lbrace
        (x,t)\in \R^n\times\R\mid |x-x_0|<R, \qandq t_0-R^2\leq t \leq t_0
        \right\rbrace,
    $$
    and let $g\in C^\infty(D,\R^+)$ be a subsolution of the nonhomogeneous heat equation with linear source:
    $$
    \sH(g)\leq Cg.
    $$
    Then 
    $$
    g(x_0,t_0)\lesssim \frac{1}{R^{n+2}}\int_{\Sigma_R}g.
    $$
\end{lemma}

\begin{proposition}
\label{prop: kappa -> 0}
    If a solution to \eqref{eq: HSF} exists for all $t<\rT$ and it has $\sup_M\epsilon(t)$ bounded, uniformly in $t$, then the following hold:
    \begin{enumerate}[(i)]
    \item
        All the derivatives $\nabla^m\tau^\cV(t)$, $m\geq1$, are  uniformly bounded in $t$.
    \item 
        If $\rT=\infty$, then $\kappa\underset{t\to\infty}{\longrightarrow}0$ uniformly in $t$.
        \end{enumerate}
\begin{proof}\quad
\begin{enumerate}[(i)]
    
    \item
    Let $\kappa_m(t):=\frac{1}{2}|\nabla^m\tau^\cV(t)|^2$. Under our assumption of a uniform energy density bound, we assert that $\kappa_m$ are subsolutions of the heat equation with linear source:
    $$
    \sH(\kappa_m)\leq C_m \kappa_m,
    $$
    for constants $C_m>0$. That indeed implies the claim, by a standard argument comparing $\kappa_m$ to a solution of the corresponding heat equation, with same initial condition, via the maximum principle.
    
    For each integer $m$, we compute:
    $$
    \nabla^m \nabla_k \nabla_k \tau^\cV  
    = \nabla_k \nabla_k \nabla^m \tau^\cV  
      + \sR_m(\tau^\cV ,\nabla \tau^\cV ,\dots,\nabla^{m}\tau^\cV )
    $$
    where $\sR_m$ is an algebraic function  with coefficients determined by the geometry of $(M,g)$, we deduce that 
    \begin{align*}
        \Delta \kappa_m 
        &= \left\langle \Delta \nabla^m \tau^\cV  , \nabla^m \tau^\cV  \right\rangle + \left\vert \nabla \nabla^m \tau^\cV \right\vert^2 \\
        &= \left\langle 
          \nabla^m \Delta \tau^\cV  , \nabla^m \tau^\cV  \right\rangle 
        + \left\langle 
          \sR_m(\tau^\cV ,\nabla \tau^\cV ,\dots,\nabla^{m}\tau^\cV ) , \nabla^m \tau^\cV  \right\rangle \\
        &\quad + |\nabla \nabla^m \tau^\cV |^2,\\ 
        \partial_t \kappa_m 
        &=\langle  \nabla^m 
        \partial_t \tau^\cV  , \nabla^m \tau^\cV  \rangle.
\end{align*}
    On the other hand, setting $\Sigma(x,t) = \sigma_{t}(x)$ as a map on $M_\rT$ along the flow \eqref{eq: HSF}, we have
    \begin{align*}
        \Delta \partial_t \sigma 
        &= \nabla^{\Sigma}_{e_i} \nabla^{\Sigma}_{e_i} d\Sigma(\partial_t) 
        = \nabla^{\Sigma}_{e_i} \left\{ \nabla^{\Sigma}_{\partial_t} d\Sigma(e_i) + d\Sigma([e_i,\partial_t]) \right\} \\
        &= \nabla^{\Sigma}_{\partial_t} \nabla^{\Sigma}_{e_i} d\Sigma(e_i) + R^{\R\times M}(e_i, \partial_t) d\Sigma(e_i) + \nabla^{\Sigma}_{e_i} (d\Sigma([e_i,\partial_t]))\\
        &= \nabla^{\Sigma}_{\partial_t} \nabla^{\Sigma}_{e_i} d\Sigma(e_i)\\
        &= \partial_t \tau^\cV. 
    \end{align*}
    Therefore
    \begin{align*}
        (\partial_t - \Delta)  \kappa_m &=  - \left\langle \sR_m(\tau^\cV ,\nabla \tau^\cV ,\dots,\nabla^{m}\tau^\cV ) , \nabla^m \tau^\cV  \right\rangle -  \left\vert \nabla \nabla^m \tau^\cV \right\vert^2.
    \end{align*}
    \item 
    This is a direct consequence of the parabolic cylinder estimate of Lemma \ref{lemma: parabolic cylinder}, by exhaustion of the initial energy $E_0$. Let us suppose not, i.e., that there exist $x\in M$, $\delta>0$ and $t_n\nearrow^\infty$ such that
    $$
    \kappa(x,t_n)>\delta, \quad \forall n\in\N. 
    $$
    By Proposition \ref{prop: heat eqs of e and k}, $\kappa$ is bounded on the unit-length cylinder, so:
    \begin{align*}
        0<\delta 
        &<\kappa(x,t_n)\leq C\int_{\Sigma_1(x,t_n)}\kappa
	    = C\int_{t_n-1}^{t_n}(\int_{B_1(x)}\kappa(\cdot,\hat t))d\hat t
  	    \leq C\int_{t_n-1}^{t_n} K(\hat t)d\hat t\\
        &= \frac{C}{2}\left( E(t_{n-1})-E(t_n)\right),
    \end{align*}
    again by Lemma \ref{lemma: E(t) unif bounded}. Hence $E(t_n)<E_0-n\frac{2\delta}{C}$, which would lead to negative energy, for $n\gg0$.
    \qedhere
\end{enumerate}
\end{proof}
\end{proposition}

\begin{corollary}
\label{cor: subseq sol to HSF under C^0-bound}
    In the context of Proposition \ref{prop: kappa -> 0}--(i), if $\rT=\infty$, there exists a sequence $t_n\nearrow^\infty$ along which the solution to the flow \eqref{eq: HSF}  converges to a smooth section $\sigma_\infty\in\Gamma(\pi)$, defining a harmonic geometric structure:
    $$
    \tau^\cV(\sigma_\infty)=0.
    $$
\begin{proof}
    Differentiating successively, we have:
    \begin{align*}
        |\nabla^m\Delta\epsilon| 
        &\leq \sum_{i=0}^m {m\choose i}
        \left\vert\left\langle 
        \nabla^id\tau^\cV,\nabla^{m-i}T
        \right\rangle\right\vert
        +\sR_m(|\tau^\cV|,|T|,|\nabla T|,\dots,|\nabla^{m+1}T|)\\
        &\lesssim \sum_{i,j=0}^{m+1}
        \left\langle
        |\nabla^{i}\tau^\cV|^2+|\nabla^{j}T|^2
        \right\rangle,
    \end{align*}    
    where $\sR_m$ is an algebraic function with coefficients determined by the Riemannian geometry of $(M,g)$, and we used Cauchy-Schwartz at the last step.
    
    We deduce the following estimates by induction on $m$, restricting ourselves to the sequence $\{t_n\}$ from Proposition \ref{prop: tauV in L^2_1}--(ii), which ensures the first step $m=0$, then using elliptic regularity and the uniform bounds on all derivatives of $\tau^\cV$ from Proposition \ref{prop: kappa -> 0}--(i):
    $$
    \Vert \epsilon(t_n)\Vert_{L^2_{m+2}}
    \leq \Vert \Delta\epsilon(t_n)\Vert_{L^2_{m}}
    \leq \Vert \tau^\cV(t_n)\Vert_{L^2_{m+1}}
    \quad\forall m\geq0.
    $$
    Hence,  $\Vert T(t_n)\Vert_{C^\infty(M)}$ is also bounded, and $\sigma_\infty:=\lim \sigma(t_n)$ is smooth.
\end{proof}
\end{corollary}

\begin{remark}
\label{rem: homogeneous strs and Lauret}
    In particular, when $M=K/L$ is itself a homogeneous manifold, then we should expect long-time existence of the flow among homogeneous geometric structures, since in that case the $L^2$-norm $E(t)$ and the pointwise density $\epsilon(t)$ are proportional and the former is always uniformly bounded, by Lemma \ref{lemma: E(t) unif bounded}. The Gromov-Hausdorff limit in this case can exhibit some quite non-trivial behaviour, following the theory by J. Lauret \cites{Lauret2012,Lauret2016}.
\end{remark}

\subsubsection{Energy density $L^m$--bounded for all time}

By a standard property of the heat kernel, the condition of uniformly bounded energy density can actually be slightly weakened into a sufficient uniform $L^p$-bound, which can be shown to be exactly $p=m:=\dim M$. 

The following instance of \cite{Grigoryan1997}*{Theorem 1.1} stems from a long series, going back to Nash (1958) and Aronson (1971) [op.cit.], of generalised `Gaussian' upper bounds in terms of the geodesic distance $r$, for the heat kernel $\rH_t$ of a Riemannian manifold: 
\begin{theorem} 
\label{Thm Gaussian bound on heat kernel}
    Let $M$ be an arbitrary connected Riemannian $m-$manifold, $x,y\in M$ and $0\leq \rT\leq\infty$; if there exist suitable [see below] real functions $f$ and $g$ such that the heat kernel $\rH_t$ satisfies the `diagonal' conditions
    $$
    \rH_t(x,x)\leq\frac{1}{f(t)}
    \qandq                    \rH_t(y,y)\leq\frac{1}{g(t)},
    \quad     
    \forall t\in \;]0,\rT[,
    $$
    then, for any $C>4$, there exists $\delta=\delta(C)>0$ such that
    $$
    \rH_t(x,y) \leq 
    \frac{(cst.)}{\sqrt{f(\delta t)g(\delta t)}} \exp\left\{-\frac{r(x,y)^2}{Ct}\right\},       \quad \forall t\in \;]0,\rT[
    $$
    where $(cst.)$ depends on the Riemannian metric only. 
\end{theorem}

For the present purposes one may assume simply $f(t)=g(t)=t^{\tfrac{m}{2}}$, but in fact $f$ and $g$ can be \emph{much} more general [Op. cit. p.37].
In particular, 
$$
\Vert \rH_t(x,\cdot)\Vert_{L^p(M)}\lesssim \frac{1}{t^{\frac{m}{2}}}
	\left(\int_M
	\exp\left\lbrace
	-p\frac{r(x,\hat{x})^2}{Ct}
    \right\rbrace d\hat{x}
    \right)^\frac{1}{p}.
$$
For $q\gg0$, we may passing from $L^{q}$ to $C^0$-bounds for subsolutions of the nonhomogeneous heat equation:
\begin{proposition}
\label{prop: C^0 bounds from L^p for subsolutions}
    Let $\sH:=\partial_t -\Delta$ denote the heat operator, and let $P(g)$ be a polynomial of degree $k$; then a heat subsolution $g$ satisfying
    $
    \sH(g)\leq P(g)
    $
    is bounded in $C^0(M_\rT)$, provided it is bounded in $L^{\frac{km}{2}}(M_\rT)$.
\begin{proof}
    By Young's convolution inequality, we have
\begin{align*}
    g(t)
    &\lesssim \int_0^t 
    |\rH_{t-\hat t}\ast P(g(\hat t))|d\hat{t}\\
	&\lesssim \int_0^t 
	\Vert \rH_{\hat{t}} \Vert_{L^p(M)} \Vert P(g(\hat{t}))\Vert_{L^{p^*}(M)}d\hat{t},
	\qforq \frac{1}{p}+\frac{1}{p^*}=1.
\end{align*}
On a complete  $m-$dimensional Riemannian manifold $M$, the heat kernel $\rH_{t}$ satisfies 
\cite[\S 9]{Eells1964} the diagonal condition of \emph{Theorem \ref{Thm Gaussian bound on heat kernel}}:
\begin{equation*}
    \rH_{t}\left( x,x\right) 
    \lesssim \frac{1}{t^{\frac{m}{2}}},\quad \forall
    x\in M.
\end{equation*}%
So, fixing $C>4$ and denoting by $r(\cdot,\cdot) $ the geodesic distance, we have 
\begin{equation*}
    \rH_{t}\left( x,\hat{x}\right) 
    \lesssim \frac{1}{t^{\frac{m}{2}}} 
    \exp \left\{ 
        -\frac{r(x,\hat{x}) ^{2}}{Ct}
        \right\},
    \qquad \forall x,\hat{x}\in M.
\end{equation*}%
Hence, for each $x\in M$,
\begin{align*}
    \left\Vert \rH_{t}\left( x,.\right) \right\Vert _{L^{p} \left( M\right) }
    &\lesssim t^{-\frac{m}{2}}
    \left( \int_{M}
    \exp \left\{ -p\frac{r(x,\hat{x}) ^{2}}{Ct}\right\} d\hat{x}\right) ^{\frac{1}{p}} \\
    &\lesssim t^{-\frac{m}{2}}
    \left( 
        \int_{0}^{\infty}
        \exp \left\{-\frac{p}{Ct}\hat{r}^2\right\} \hat{r}^{m-1}d\hat{r}
    \right) ^{\frac{1}{p}}\\
    &\lesssim t^{-\frac{m}{2}}
    \left( 
        \int_{0}^{\infty}
        \left( \frac{Ct}{p}\right)^{\frac{m}{2}} \exp\{-\hat{u}^{2}\} 
        \hat{u}^{m-1}d\hat{u}
    \right) ^{\frac{1}{p}} \\
    &\lesssim 
    t^{\frac{m}{2}
    \left( \frac{1}{p}-1\right)}. 
\end{align*}%
Now, 
$$  
    \frac{m}{2}
    \left( \frac{1}{p}-1\right) >-1 
    \quad\Leftrightarrow\quad
    p< \frac{m}{m-2}
    \quad\Leftrightarrow\quad 
    p^*>\frac{m}{2},    
$$
in which case, for some constant $c_{p}(\rT)>0$ depending on the diameter of $(M,g)$, 
\begin{equation*}
    \int_{0}^{\rT}\left\Vert \rH_{t}( x,\cdot) \right\Vert _{L^{p} }dt\leq c_{p}(\rT).
    \qedhere
\end{equation*}
\end{proof}
\end{proposition}

\begin{corollary}
\label{cor: subseq sol to HSF under L^n-bound}
If $\epsilon\in L^m(M_\rT)$, for $m=\dim_\R M$, then actually $\epsilon\in C^0(M_\rT)$.
\begin{proof}
Use Gaussian bounds on the heat kernel: 
$$
\epsilon(t)\lesssim 1+\int_0^t \Vert\rH_{\hat t}\Vert_{L^p}
\Vert C_1\epsilon+C_2\epsilon^2\Vert_{L^{p^*}}d\hat t
$$
This implies the assertion, provided $\epsilon\in L^{2p^*}(M_\rT)$.
\end{proof}
\end{corollary}

Together, Corollaries \ref{cor: subseq sol to HSF under C^0-bound} and \ref{cor: subseq sol to HSF under L^n-bound} prove Theorem \ref{thm: HSF under bounded torsion}.

\subsubsection{Long-time existence from small initial energy}
\label{sec: small energy -> T=0}

We introduce  the following small adaptation of a fundamental result by Chen-Ding, which establishes long-time existence of the HMHF for sufficiently small initial energy:

\begin{proposition}[{\cite[Corollary 1.1]{Chen-Ding1990}}]
\label{prop: Chen-Ding small energy}
    Let $(P,\eta)$ and $(Q,\tilde\eta)$ be compact Riemannian manifolds without boundary, $\dim P \geq 3$. For each $c >0$, there exists a constant $\re(c)>0$ such that, if $s_0 \in C^{2,\alpha}(P,Q)$  satisfies
    \begin{enumerate}[(i)]
        \item 
        $|ds_0 (z)| \leq c$, $\forall z \in \P$,
        \item 
        $E(s_0) < \re(c)$,
    \end{enumerate}
\noindent 
    then there exists a solution $\{s_t\}$ to the harmonic map heat flow \eqref{eq: HMHF} with initial data $s_0$, defined for all $t> 0$. Moreover $s(t)$ converges to a constant map as $t\to\infty$.
\end{proposition}

\begin{remark}
\label{rem: unfortunate Chen-Ding}
    Since the Dirichlet energies of a section $\sigma$ and its corresponding $G$-equivariant map $s$ are proportional, up to the constants from Lemma \ref{lemma: mu relates harmonicities}, at this point one might be led to ask whether, perhaps in a \emph{very optimistic} context, the small-energy condition goes over from the former to the latter, thus satisfying the hypotheses of Proposition \ref{prop: Chen-Ding small energy}. In other words, whether there exists a uniform constant $\rc_0 >0$ such that, if  the initial condition $\sigma_0 \in C^{\infty}(M,N)$ satisfies $\bar\epsilon_0 \leq \rc_0$, then there exists a unique smooth solution to the flow \eqref{eq: HSF} for all $t\geq 0$. Notice that if that was indeed the case, then we know from Lemma \ref{lem: various torsions} that constant $G$-equivariant maps are precisely the lifts of torsion-free sections, so the limit would be a torsion-free structure.
     
    However, that intuition should be taken with scepticism, because quite often $E(s_0)$ is definitely not small. Looking closely at the proof of \cite[Thm 1.1 \& Cor 1.1]{Chen-Ding1990}, on one hand, there exists a uniform small constant $\delta_0>0$, such that a maximum time must be at least
    $$
    \rT>\delta_0 \ln \left(1+\frac{1}{2c^2}\right).
    $$
    On the other hand, denoting by $\rho>0$ the injectivity radius of $(P,\eta)$, there exists a uniform constant $\gamma>0$ such that the following implication holds:
$$
    E(s_0)<\gamma \rho^{m-2} 
    \Rightarrow
    \rT<\rho^2.
    $$
    Together with the previous inequality, this implies
$$
    c>\frac{1}{\sqrt{2\exp{\frac{\rho^2}{\delta}}-1}}.
    $$
    Moreover, $E(s_0)=\ra_P.E(\sigma_0)+\rb_P$, so set
    $\rc_0:=\frac{\gamma\rho^{m-2}-\rb_P}{\ra_P\vol M}$.
    It is then clear that, whenever $\gamma\rho^{m-2}>\rb_P$, the desired assertion would indeed make sense. Unfortunately, there is no immediate constructive way to determine the constant $\gamma$, which comes from a Moser iteration argument, so in practice this procedure cannot be expected to work. 
    Besides, a map $s:P\to G/H$ at once constant and equivariant might very well not exist at all, for topological reasons.
\end{remark}

\subsection{A pseudo-monotonicity formula for the harmonic section flow}

We formulate a straightforward argument leading to a pseudo-monotonicity formula for a natural entropy functional along the harmonic section flow \eqref{eq: HSF}. It likely underlies an $\epsilon$-regularity theory for $G$-equivariant maps along the lines of \cites{Chen-Ding1990,Hamilton1993,Boling2017}, which will be the object of subsequent work, see Afterword.

Let $Q$ denote the Riemannian manifolds $M$, the total space of $P$ or its fibre $G$, and set
$$
\rd_Q:=\dim(Q),
\quad
\rv_Q:=\vol(Q)
\qforq 
Q=P,M,G.
$$ 
Following \cite{Boling2017}, we define the system of flows
\begin{equation}\label{eq2.3}
    \begin{cases}
    \displaystyle\partial_t v_t &= -1 \\
    v_t|_{t=t_0} &= v_{0} 
    \end{cases}
\end{equation}
and 
\begin{equation}\label{eq2.4}
    \begin{cases}
    \displaystyle\partial_t \theta^Q_t &= (|\nabla^Q \theta^Q_t|^2 -  \Delta^Q \theta^Q_t) + (\rd_Q)/(2v_t)\\
    \theta^Q_t|_{t=t_0} &= \theta^Q_{0} 
    \end{cases}
\end{equation}
such that 
$
\Theta^Q_t = (4\pi v_t)^{-\frac{\rd_Q}{2}} e^{-\theta^Q_t}
$
is a solution of the backward heat equation on $Q$:
\begin{equation}
\label{eq: Theta Q}
    \begin{cases}
    \displaystyle\partial_t\Psi^Q_t &= - \Delta^Q (\Psi^Q_t)\\
    \Psi^Q_t|_{t=t_0} &= \Psi^Q_{t_0} .
    \end{cases}
\end{equation}
for $0\leq t <t_0$. As $\pi$ is a Riemannian submersion with minimal fibres, the solutions to \eqref{eq2.4} and \eqref{eq: Theta Q} on $P$ and $M$ are related, respectively, by
\begin{align*}
    \theta^P_t 
    &= - (\rd_G /2) \ln{v_t} + \theta^M_t\circ \pi, \\
    \Theta^P_t 
    &= (4\pi)^{-(\rd_G /2)} \Theta^M_t\circ \pi .
\end{align*}
The functionals
\begin{equation*}
    \fF^P_{s}(t) 
    = \frac{T-t}{2} \int_P  \Theta^P_t |ds_t|^2
    \qandq
    \fF^M_{\sigma}(t) 
    = \frac{T-t}{2} \int_M \Theta^M_t |d^\cV \sigma_t|^2  
\end{equation*}
are then related by
\begin{equation*}
   \fF^P_{s}(t) = (4\pi)^{-(\rd_G /2)}\rv_G \fF^M_{\sigma}(t) + \frac{T-t}{2} \rd_G
\end{equation*}
with the arguments of Lemma~\ref{lemma: mu relates harmonicities}, assuming the normalisation
$$
\int_P \Theta^P_{t_0} 
= 1 .
$$
This allows us to apply Hamilton~\cite[Theorem A]{Hamilton1993} to $\fF^P_{s}$, to have, for $T-1 \leq \tau \leq t \leq T$ and the harmonic $G$-equivariant map $s : P \to G/H$,
$$
\fF^P_{s}(t) \leq C_P \fF^P_{s}(\tau) + C_P (t - \tau) \int_P |d s_0|^2
$$
or, in terms of sections, a pseudo-monotonicity formula for $\sigma$:
\begin{proposition}
Under the assumptions \eqref{eq: conditions of HSF}, a solution $\{\sigma_t\}\subset \Gamma(\pi)$ of the flow~\eqref{eq: HSF} satisfies, for $\rT-1 \leq \tau \leq t \leq \rT$,
\begin{align*}
&(4\pi)^{\frac{\rd_G}{2}}\fF^M_{\sigma}(t) + \frac{\rT-t}{2} \frac{\rd_G}{\rv_G}   
    \leq 
    C_P \Big( (4\pi)^{-\frac{\rd_G}{2}} \fF^M_{\sigma}(\tau) +  
   \frac{\rT - \tau}{2} \frac{\rd_G}{\rv_G} 
   + (t - \tau) (E(\sigma_0)  
   + \frac{\rd_G \rv_M}{2} \Big),
\end{align*}
where $\rv_P = \rv_G\rv_M$.
\end{proposition}

\part{Instances of harmonic geometric structures}

\section{Parallelisms on the $3$-sphere}
Let us begin to illustrate the abstract framework of Part 1 in the case of parallelisms, corresponding to the simplest possible choice of the (trivial) subgroup $H=\{e\}\subset \SO(m)$. 
The outcome will be a natural harmonicity condition, together with its associated flow, which to our knowledge has not so far been explicitly studied. Qualitatively, while the very existence of a  parallelism has the topological interpretation of trivialising $TM$, harmonicity is a weaker condition than covariant constancy, and therefore has the potential to distinguish global frames for instance on manifolds known not to admit a Lie group structure.  

In this case, the bundles $P= P_{\SO(m)}=N\to M$ coincide, and their vertical and horizontal distributions are the same, since the quotient map $q$ is the identity. The Lie algebra $\fh$ is trivial, its complement is $\fm = \fg$ and the associated bundle $\ufm$ is $P \times \fg$. The canonical isomorphism becomes 
$$
f(q_* (E)) = f(E) = (z,\omega(E)) \in P \times \fg .
$$
A section $\sigma : M \to N=P$ corresponds to a global frame field, i.e. a \emph{parallelism} on $M$. Then $\omega$ is the Maurer-Cartan form of $\fg$, since $d^\cV \sigma_x (X) \in \cV_x$ and the typical fibre of $P$ is the group $G$ itself, and the connection $1$-form $f$ acts on the vertical torsion by: 
$$
f(d\sigma(X)) = \omega (d^\cV \sigma (X)) \in P \times \fg,
\qforq
X\in TM.
$$
To the best of our knowledge, there are no results in the literature regarding harmonic flows of parallelisms, therefore the statements of Theorems \ref{thm: uniqueness and s-t existence}-\ref{thm: HSF under bounded torsion} in this context seem to be original.

\subsection{Torsion   of a parallelism on $\sn^3$}
\label{sec: torsion parallelism}
To make matters concrete, let us carry out the detailed derivation of the harmonicity condition for a parallelism on $M=\sn^3$, so that $P=\SO(4)\to \sn^3$ is an $\SO(3)$-bundle. It is rather clear how to extend this procedure to $(\sn^m,g)$ with more general metrics, for $m=3,7$.

Recall that, for $z\in \SO(4)$,
\begin{equation}\label{tangent SO}
T_{z} \SO(4) = \Big\{ \frac12 ( M - z M^T z) : M\in \cM_{4} (\R) \Big\} = \{ z X : X + X^T = 0 : X\in \cM_{4} (\R)\},
\end{equation}
The projection 
$$
\begin{array}{rrcl}
    p : &\SO(4) &\to &\sn^3   \\
    &z=(x,z_1 , z_2 , z_3) &\mapsto &x
\end{array}
$$
is $\SO(4)$-equivariant, i.e. 
$p (\tilde{z} z) = \tilde{z} p(z)$, for all $\tilde{z} \in \SO(4)$, so the vertical and horizontal distributions are related by
$$ \cV_{\tilde{z}z} = \tilde{z}\cV_{z} \quad \text{and} \quad \cH_{\tilde{z}z} = \tilde{z}\cH_{z}.$$
The differential $p_*$ at the identity maps the matrix $M= (M_0 ,M_1 , M_2 , M_3)$ to $M_0$, so we deduce from the anti-symmetry in $\so(4)$ that
$$
\cV_{\id} = \Big\{ M\in \so(4) : M= 
\begin{pmatrix}
0 & 0 \\
0 & \tilde{M}
\end{pmatrix}, \tilde{M}\in \so(3)
\Big\}
\qandq
\cV_{z} = z \cV_{\id}.
$$

A section of the orthonormal frame bundle
$$
\begin{array}{rrcl}
    \sigma : &\sn^3 &\to &\SO(4) \\
    & x &\mapsto &\sigma(x) = (x,\sigma_1 (x),\sigma_2 (x),\sigma_3 (x))
\end{array}
$$ 
has, at each $x\in \sn^3$, the linearisation
$$
\begin{array}{rrcl}
    d\sigma_x \co &T_x \sn^3 &\to &T_{\sigma(x)}\SO(4)\\
    &X &\mapsto &d\sigma_x (X)=: \Proj_{T_{\sigma(x)}\SO(4)}(M)
\end{array}
$$
with $M_\sigma = (X,d\sigma_1 (X),d\sigma_2 (X),d\sigma_3 (X))$. Therefore
\begin{align*}
d\sigma_x (X) 
    &= \tfrac12 \Big( M_\sigma  - \sigma(x) M_\sigma ^T \sigma(x)\Big)
    =\sigma(x) .\tfrac12\Big(  \sigma(x)^T M_\sigma  -  M_\sigma ^T \sigma(x)\Big)\\
    &= \sigma(x) \begin{pmatrix}
0 & -v^T \\
v & \tilde{M_\sigma }
\end{pmatrix} ,
\qwithq
\tilde{M_\sigma }\in \so(3)
\qandq 
v\in \R^3,\\
\Rightarrow\quad
d^\cV \sigma_x (X) 
    &
    =\sigma(x) \begin{pmatrix}
0 & 0 \\
0 & \tilde{M_\sigma }
\end{pmatrix}.
\end{align*}

Then 
$$
\sigma(x)^T M_\sigma  =
\begin{pmatrix}
\langle x,X\rangle &  \langle x,d\sigma_1 (X)\rangle &  \langle x,d\sigma_2 (X)\rangle &   \langle x,d\sigma_3 (X)\rangle\\
\langle \sigma_1, X \rangle &  \langle \sigma_1 ,d\sigma_1 (X)\rangle &  \langle \sigma_1,d\sigma_2 (X)\rangle &   \langle \sigma_1,d\sigma_3 (X)\rangle\\
\langle \sigma_2, X \rangle &  \langle \sigma_2 ,d\sigma_1 (X)\rangle &  \langle \sigma_2,d\sigma_2 (X)\rangle &   \langle \sigma_2,d\sigma_3 (X)\rangle\\
\langle \sigma_3, X \rangle &  \langle \sigma_3 ,d\sigma_1 (X)\rangle &  \langle \sigma_3,d\sigma_2 (X)\rangle &   \langle \sigma_3,d\sigma_3 (X)\rangle\\
\end{pmatrix}
$$
and since 
\begin{align*}
    \langle x, X \rangle&=\langle \sigma_i,d\sigma_i (X)\rangle =0, 
\quad
\langle x,d\sigma_i (X)\rangle 
= -\langle \sigma_i , X\rangle\\
\langle \sigma_i,d\sigma_j (X)\rangle&=\langle \sigma_i,\nabla^{\sn^3}_{X}\sigma_j\rangle, 
\qforq i,j=1,2,3,
\end{align*}
we obtain, cf. \eqref{tangent SO}: 
\begin{align*}
    d\sigma_x (X) &=
    \sigma(x)
    \begin{pmatrix}
    0 &  -\langle \sigma_1 , X\rangle &  - \langle \sigma_2 , X\rangle &   - \langle \sigma_3 , X\rangle\\
    \langle \sigma_1 , X\rangle&  0 &  \langle \sigma_1,\nabla^{\sn^3}_{X}\sigma_2\rangle &   \langle \sigma_1,\nabla^{\sn^3}_{X}\sigma_3\rangle\\
    \langle \sigma_2 , X\rangle&  \langle \sigma_2,\nabla^{\sn^3}_{X}\sigma_1\rangle &  0 &   \langle \sigma_2,\nabla^{\sn^3}_{X}\sigma_3\rangle\\
    \langle \sigma_3 , X\rangle &  \langle \sigma_3,\nabla^{\sn^3}_{X}\sigma_1\rangle &  \langle \sigma_3,\nabla^{\sn^3}_{X}\sigma_2\rangle &   0\\
\end{pmatrix}\in T_{\sigma(x)}\SO(4)\\
    \Rightarrow\quad
    d^\cV\sigma_x (X) 
    &=\sigma(x)
    \begin{pmatrix}
    0 &  0 &0 &0 \\
    0 &  0 &  \langle \sigma_1,\nabla^{\sn^3}_{X}\sigma_2\rangle &   \langle \sigma_1,\nabla^{\sn^3}_{X}\sigma_3\rangle\\
    0 &  \langle \sigma_2,\nabla^{\sn^3}_{X}\sigma_1\rangle &  0 &   \langle \sigma_2,\nabla^{\sn^3}_{X}\sigma_3\rangle\\
    0 &  \langle \sigma_3,\nabla^{\sn^3}_{X}\sigma_1\rangle &  \langle \sigma_3,\nabla^{\sn^3}_{X}\sigma_2\rangle &   0\\
\end{pmatrix}
\end{align*}
Then, if $\omega$ is the Maurer-Cartan form of $\fg=\so(3)$ and $\tilde\sigma(x) = (\sigma_1 (x),\sigma_2 (x),\sigma_3 (x))$, then
\begin{align}
    \cI (d^\cV\sigma_x (X)) &= \omega (d^\cV\sigma_x (X)) \notag 
     = \tilde\sigma(x) (d^\cV\sigma_x (X)) \notag \\
    &= \Bigg( \sigma(x) , 
    \begin{pmatrix}
    \label{eq: matrix}
        0 &  \langle \sigma_1,\nabla^{\sn^3}_{X} \sigma_2\rangle &   \langle \sigma_1,\nabla^{\sn^3}_{X} \sigma_3\rangle\\
        \langle \sigma_2,\nabla^{\sn^3}_{X} \sigma_1\rangle &  0 &   \langle \sigma_2,\nabla^{\sn^3}_{X} \sigma_3\rangle\\
        \langle \sigma_3,\nabla^{\sn^3}_{X} \sigma_1\rangle &  \langle \sigma_3,\nabla^{\sn^3}_{X} \sigma_2\rangle &   0\\
    \end{pmatrix} \Bigg) \in P \times \so(3).
\end{align}

\subsection{Harmonic parallelisms on $\sn^3$}
Recall that $\nabla^{\omega}$ is the connection on $T^* \sn^3 \times T \sn^3$, i.e. the connection on endomorphisms of $\sn^3$. Since $p^* \ufg = \ufm \oplus \ufh$, where $\ufh = \P\times_H \fh$, we identify $ \cI (d^{\cV}\sigma_x (X))$ with the endomorphism of $\sn^3$ defined by the matrix of \eqref{eq: matrix} in the frame $\tilde\sigma(x)$.
Since $\cI$ only sees the $\cV$-component, by Lemma~\ref{lemma: 2nd ff} we have:
$$ 
\cI ((\nabla^{\cV}d^{\cV}\sigma)(X,X)) = ( \nabla^{\omega} (\sigma^* f)) (X,X).
$$
Specialising to vectors of the frame itself,
$$
\cI (d^\cV\sigma_x (X)) (\sigma_i) 
= (\sigma^* f) (X) = \nabla^{\sn^3}_{X} \sigma_i, 
\qforq i=1,2,3.
$$
If $X$ is a vector field on $\sn^3$ such that $\nabla^{\sn^3}_{X} X =0$, then
$$
( \nabla^{\omega} (\sigma^* f)) (X,X) = \nabla^{\omega}_X ((\sigma^* f) (X))
\in T^* \sn^3 \times T \sn^3,
$$
and its evaluation on the column-vectors of $(\sigma_1 (x),\sigma_2 (x),\sigma_3 (x))$ yields
\begin{align*}
(\nabla^{\omega}_X (\sigma^* f) (X)) (\sigma_i) &= \nabla^{\sn^3}_X ((\sigma^* f) (X) (\sigma_i)) - (\sigma^* f) (X) (\nabla^{\sn^3}_X \sigma_i) \\
&= \nabla^{\sn^3}_X \nabla^{\sn^3}_{X} \sigma_i - \langle \nabla^{\sn^3}_X \sigma_i , \sigma_j\rangle \nabla^{\sn^3}_X \sigma_j  - 
\langle \nabla^{\sn^3}_X \sigma_i , \sigma_k\rangle \nabla^{\sn^3}_X \sigma_k
\end{align*}
for $\{i,j,k\} =\{ 1,2,3\}$. Taking traces, 
$$
\cI (\tau^{\cV} (\sigma) ) (\sigma_i) = 
 \nabla^*\nabla \sigma_i - \tr \langle \nabla^{\sn^3}_\bullet  \sigma_i , \sigma_j\rangle \nabla^{\sn^3}_\bullet  \sigma_j - \tr \langle \nabla^{\sn^3}_\bullet \sigma_i , \sigma_k\rangle \nabla^{\sn^3}_\bullet  \sigma_k ,
$$
taking $\nabla^*\nabla \sigma_i = \tr \nabla^2 \sigma_i$. One can easily check that
\begin{align*}
\langle \cI (\tau^{\cV} (\sigma) ) (\sigma_i), \sigma_i \rangle = 0 \qandq 
\langle \cI (\tau^{\cV} (\sigma) ) (\sigma_i), \sigma_j \rangle = - \langle \cI (\tau^{\cV} (\sigma) ) (\sigma_j), \sigma_i \rangle, 
\end{align*}
so $\cI (\tau^{\cV} (\sigma) )$ is indeed in $\sigma^*(P\times_{\SO(3)} \so(3))$.
One can re-write
\begin{align}
\label{eq: tension of frame on S^3}
    \cI (\tau^{\cV} (\sigma) )(\sigma_i) &=  \nabla^*\nabla \sigma_i + |\nabla^{\sn^3} \sigma_i|^2 \sigma_i 
    - \tr \langle \nabla^{\sn^3}_\bullet  \sigma_i , \sigma_k\rangle  \langle \nabla^{\sn^3}_\bullet  \sigma_k , \sigma_j \rangle \sigma_j 
    - \tr \langle \nabla^{\sn^3}_\bullet  \sigma_i , \sigma_j\rangle  \langle \nabla^{\sn^3}_\bullet  \sigma_j , \sigma_k \rangle \sigma_k \notag \\
    &=  \nabla^*\nabla \sigma_i + |\nabla^{\sn^3} \sigma_i|^2 \sigma_i 
    + \langle \nabla^{\sn^3}  \sigma_i , \nabla^{\sn^3} \sigma_j \rangle \sigma_j 
    + \langle \nabla^{\sn^3}  \sigma_i , \nabla^{\sn^3} \sigma_k \rangle \sigma_k.
\end{align}
Recalling our curvature convention \eqref{eq: curvature convention} and since $\nabla^*\nabla \sigma_i = \tr \nabla^2 \sigma_i$, it is not hard to check that eg. the Hopf vector fields 
$(\mathrm i x,\mathrm j x, \mathrm k x)$ are a solution, since they are unit harmonic vector fields. Indeed,  they are Killing vector fields satisfying $\nabla^{\sn^3}_{\sigma_i}  \sigma_i =0$ and $\nabla^{\sn^3}_{\sigma_i}  \sigma_j = \pm \sigma_k$, for $i\neq j \neq k \neq i$, and $\Ric^{\sn^3} = 2\id$ \cite{Wiegmink1995,Wood2000}. 

More generally, if  $\sigma_1$, $\sigma_2, \sigma_3$ are orthogonal unit harmonic vector fields on ${\sn^3}$, then the orthonormal frame $\sigma=\{\sigma_1, \sigma_2, \sigma_3\}$ is a harmonic parallelism if, and only if, the $(1,1)$-tensors $\nabla^{\sn^3} \sigma_i$ and $\nabla^{\sn^3} \sigma_j$ are mutually orthogonal.

\subsection{The harmonic parallelism flow from \eqref{eq: HSF}}

A natural harmonic flow for parallelisms eg. on $\sn^3$ stems from applying the canonical isomorphism of Section~\ref{sec: canonical geom on hfb} to both sides of the equation~\eqref{eq: HSF}:
$$
 \cI(\partial_t\sigma) = \cI(\tau^\cV (\sigma)).
$$
While $\cI(\tau^\cV (\sigma))$ was computed in \eqref{eq: tension of frame on S^3}, in order to deduce $\cI(\partial_t\sigma)$ from Section~\ref{sec: torsion parallelism}, we extend the $G$-action trivially on $I=[0,\rT[$  and apply the constructions from Sections~\ref{sec: canonical geom on hfb} and \ref{Section 1.2} to $M_\rT:=M\times I$. Define
\begin{align*}
    \Sigma : M_\rT 
    &\to N \\
    (x,t) 
    &\mapsto  \Sigma(x,t) = \sigma_t (x) ,
\end{align*}
so that $d\Sigma(\partial_t) = \partial_t \sigma$.
Since the initial data $\sigma_0$ of \eqref{eq: HSF} is a section, $\sigma_t$ remains a section for all $t\in I$. Hence $(\pi\circ\Sigma) (x,t) =x$ and, in particular,
$$ 
d\pi(d\Sigma(\partial_t)) =0 ,
$$
i.e. $d\Sigma(\partial_t)$ is vertical, and therefore lies in the domain of $\cI$.

The formulas of Section \ref{sec: torsion parallelism} extend to $M\times I$, $\Sigma$ and $\partial_t$ (cf. Proposition~\ref{prop: HMHF}) and we deduce that
\begin{align*}
    \cI (d\Sigma_{(x,t)} (\partial_t)) &= \cI (\partial_t \sigma) \notag \\
    &= \Bigg( \Sigma(x,t) , 
    \begin{pmatrix}
        0 &  \langle \sigma_1,\partial_t \sigma_2\rangle &   \langle \sigma_1,\partial_t \sigma_3\rangle\\
        \langle \sigma_2,\partial_t \sigma_1\rangle &  0 &   \langle \sigma_2,\partial_t \sigma_3\rangle\\
        \langle \sigma_3,\partial_t \sigma_1\rangle &  \langle \sigma_3,\partial_t \sigma_2\rangle &   0\\
    \end{pmatrix} \Bigg).
\end{align*}
The harmonic flow for parallelisms of $\sn^3$ is therefore given by the following matrix equation in $\so(3)$:
\begin{align*}
   & \begin{pmatrix}
        0 &  \langle \sigma_1,\partial_t \sigma_2\rangle &   \langle \sigma_1,\partial_t \sigma_3\rangle\\
        \langle \sigma_2,\partial_t \sigma_1\rangle &  0 &   \langle \sigma_2,\partial_t \sigma_3\rangle\\
        \langle \sigma_3,\partial_t \sigma_1\rangle &  \langle \sigma_3,\partial_t \sigma_2\rangle &   0\\
    \end{pmatrix} 
    =\\
    &\begin{pmatrix}
        0 &  \langle \nabla^*\nabla \sigma_1 , \sigma_2\rangle 
    +  \langle \nabla^{\sn^3}  \sigma_1 , \nabla^{\sn^3}\sigma_2\rangle  &   \langle \nabla^*\nabla \sigma_1 , \sigma_3\rangle 
    +  \langle \nabla^{\sn^3}  \sigma_1 , \nabla^{\sn^3}\sigma_3\rangle \\
        \langle \nabla^*\nabla \sigma_2 , \sigma_1\rangle 
    +  \langle \nabla^{\sn^3}  \sigma_1 , \nabla^{\sn^3}\sigma_2\rangle
    &  0 &  
        \langle \nabla^*\nabla \sigma_2 , \sigma_3\rangle 
    +  \langle \nabla^{\sn^3}  \sigma_2 , \nabla^{\sn^3}\sigma_3\rangle\\
        \langle \nabla^*\nabla \sigma_3 , \sigma_1\rangle 
    +  \langle \nabla^{\sn^3}  \sigma_1 , \nabla^{\sn^3}\sigma_3\rangle &   \langle \nabla^*\nabla \sigma_3 , \sigma_2\rangle 
    +  \langle \nabla^{\sn^3}  \sigma_2 , \nabla^{\sn^3}\sigma_3\rangle &   0\\
    \end{pmatrix} 
    ,
\end{align*}
or re-written in terms of the individual (column) vector fields:
\begin{equation*}
    \partial_t \sigma_i =  \nabla^*\nabla \sigma_i + |\nabla^{\sn^3} \sigma_i|^2 \sigma_i 
    + \langle \nabla^{\sn^3}  \sigma_i , \nabla^{\sn^3} \sigma_j \rangle \sigma_j 
    + \langle \nabla^{\sn^3}  \sigma_i , \nabla^{\sn^3} \sigma_k \rangle \sigma_k,
    \qforq
    \{i,j,k\} =\{ 1,2,3\}.
\end{equation*}
We conclude this example with an immediate consequence of Theorems  \ref{thm: uniqueness and s-t existence}-\ref{thm: HSF under bounded torsion}, which we state in the present context, but which could be adapted to more general parallelisable manifolds:
\begin{corollary}
\label{cor: bounded T => T-> infty S^ 3}
    Let $\sigma_0:\sn^3\to\SO(4)$ be a  parallelism on the Euclidean $3$-sphere. Suppose the torsion \eqref{eq: matrix} is uniformly $L^{6}(M)$-bounded along the harmonic section flow \eqref{eq: HSF}, with initial condition $\sigma_0$. Then the harmonic parallelism flow admits a \emph{continuous} solution $\sigma(t)$ for all time. Moreover, there exists a strictly increasing sequence $\{t_j\}\subset \left[0,+\infty\right[$ 
    such that $\sigma(t_j)
    \overset{C^\infty}{\longrightarrow} (\sigma_\infty:\sn^3\to\SO(4))$, and the limiting parallelism satisfies the harmonicity condition:
    $$
    \nabla^*\nabla \sigma_i + |\nabla^{\sn^3} \sigma_i|^2 \sigma_i 
    + \langle \nabla^{\sn^3}  \sigma_i , \nabla^{\sn^3} \sigma_j \rangle \sigma_j 
    + \langle \nabla^{\sn^3}  \sigma_i , \nabla^{\sn^3} \sigma_k \rangle \sigma_k=0,
    \qforq
    \{i,j,k\} =\{ 1,2,3\}.
    $$
\end{corollary}


\section{Almost contact structures}

Let $(M^ {2n+1},g)$ be an odd-dimensional Riemannian manifold. An \emph{almost contact structure} $(\xi,\theta)$ on $M$ is defined by a unit vector field $\xi\in\Gamma(TM)$ and a $(1,1)$-tensor $\theta\in\End(TM)$ related by 
\begin{equation} 
    \theta^2 
    = -\id_{TM} + \eta\otimes \xi ,
\end{equation}
where $\eta\in\Omega^1(M)$ is determined by $\eta(\xi)=1$ and metric compatibility is taken as a blanket assumption~\cites{Gray1959,Blair2010}.
The distribution orthogonal to the $\xi$-direction will be denoted by $\cD$, and the component of $\theta$ transverse to $\xi$ by 
$$
J:=\theta|_\cD.
$$ 
The induced connection and curvature on $\cD$ will be denoted by $\bar{\nabla}$ and $\bar{R}$, respectively.
Contact structures are characterised by the condition $\eta \wedge (d\eta)^n \neq 0$, i.e. the contact sub-bundle $\cD$ is maximally non-integrable.

\subsection{Torsion of an almost contact structure}

The general approach of harmonicity via reduction of the structure group applies to almost contact structures \cite{Vergara-Diaz2006}, for $H=\mathrm{U}(n)\subset\mathrm{SO}(2n+1)$, embedded by 
$$
A+\bi B \mapsto \begin{pmatrix}
A & -B & 0 \\
B & A & \vdots \\
0 & \cdots & 1
\end{pmatrix}.
$$
Then $\U(n) = \{ A\in \SO(2n+1) : A\phi_0 A^{-1} = \phi_0\}$, with 
$$\phi_0 := 
\begin{pmatrix}
{\mathbb O}_{n} & - {\mathbb I}_{n} & 0 \\
{\mathbb I}_{n} & {\mathbb O}_{n} & \vdots \\
0 & \cdots & 0
\end{pmatrix}
\in\fg=\so(2n+1),
$$
where ${\mathbb O}_{n}$ and ${\mathbb I}_{n}$ are the $n\times n$ zero and identity matrices, respectively. Its Lie algebra is $$
\fh = \fu(n) = \{ a\in \fg : [ a,\phi_0] = 0\}.
$$ 
The orthogonal complement of $\fu(n)$ in $\so(2n+1)$, with respect to the Killing form, splits into $\fm_1$ and $\fm_2$:
$$
\fm_1 =\{ a\in \so(2n+1) : \{a,\phi_0\} =0\}, \quad \fm_2 = \{ \{a,\eta_0 \otimes \xi_0 \} : a\in \so(2n+1)\},
$$
where $\xi_0 = (0,\dots,0,1) \in \R^{2n+1}$ and $\eta_0$ is the dual of $\xi_0$. Then $\so(2n+1) = \fu(n) \oplus \fm_1 \oplus \fm_2$ is an $\mathrm{Ad}(\U(n))$-invariant splitting with:
$$ 
[\fu(n) , \fm_i ] \subset \fm_i 
\quad (i=1,2), 
\quad [\fm_1 ,\fm_1 ] \subset \fu(n) , 
\quad [\fm_2 ,\fm_2 ] \subset \fu(n) \oplus \fm_1, 
\quad [\fm_1 ,\fm_2 ] \subset \fm_2 .
$$
If $a\in \so(2n+1)$ then $a = a_{\fu(n)} + a_{\fm_1} + a_{\fm_2}$ with
$$
a_{\fu(n)} = -\tfrac{1}{2} ( \phi_0 \{ a,\phi_0 \} + a \circ (\eta_0 \otimes \xi_0)) ,
$$
$$ 
a_{\fm_1} = \tfrac{1}{2} ( \phi_0 [ a,\phi_0 ] - a \circ (\eta_0 \otimes \xi_0)) ,
\quad a_{\fm_2} = \{ a , \eta_0 \otimes \xi_0 \},
$$
and $\xi_0$ and $\phi_0$ induce a universal almost contact structure defined by $\zg$ and $\pg$ on $\dg = \pi^* TM$.

The $\U(n)$-modules $\fu(n)$, $\fm_1$ and $\fm_2$ are respectively the typical fibres of the vector bundles 
$\underline{\fu(n)}$, ${\ufm}_1$ and ${\ufm}_2$, associated to $q : P \to N$,
isomorphic to the sub-bundles of $\Skew\dg \to N$ which, respectively, commute with $\pg$, anti-commute with $\pg$ and interchange $\im{\pg}$ and $\langle \zg\rangle$. 
Then, the homogeneous connection form is the ${\ufm}_1\oplus{\ufm}_2$-valued $1$-form $f$ on $N$ obtained by projecting the ($\fm_1 + \fm_2$)-component of the connection form $\omega \in \Omega^{1}(P,\fg)$ and its ${\ufm}_1$- and ${\ufm}_2$-components are \cite{Vergara-Diaz2006}*{Lemma 2.1}:

\begin{align*}
f_1 &= \tfrac{1}{2} \pg \circ (\nabla \pg )_1 ; \quad
f_2 = [ \pg , (\nabla \pg )_2 ].
\end{align*}

$$\xymatrix{
\llap{$v_z\in$ }T\P 
	\ar[dr]^{q_*} 
    \ar[dd]_{p_*} 
    \ar[r]^-\omega 
    \ar@/^2pc/[r]!<5mm,0mm>^{\omega_{\fm}}
    & \so(2n+1) = (\fm_1\oplus \fm_2) \ar@/^2pc/[r]^{z\bullet} \oplus \fu(n) & {\ufm_1\oplus \ufm_2} \\ 
  &TN 
  	\ar[dl]_{\pi_*} 
    \ar@{=}[r]
    \ar[ur]_{f_1\oplus f_2} 
    & \cV \rlap{ $\oplus\;\cH$}
  		\ar[u] _{\cI_1\oplus\cI_2}  \\ 
TM & &     
}$$
Finally, in order to express the torsion in terms of $J$ and $\eta$, we further need to apply the `hat' isomorphisms: \cite{Vergara-Diaz2006}*{Proposition 2.2}
\begin{equation}
\label{torsion almost contact structure}
\begin{cases}
    (\cI_1 d^\cV \sigma (X))\hat{} 
    &=  \tfrac12 J \nabla_X J \\
    (\cI_2 d^\cV \sigma (X))\hat{} 
    &=  \nabla_X \xi .
\end{cases}
\end{equation}

\subsection{Harmonic almost contact structures}

From \cite{Wood2003}*{(3.2)}, the canonical vector bundle isomorphism $\cI : \cV =\ker d\pi \to {\ufm}_1 \oplus {\ufm}_2$ sends $d^{\cV}\sigma$ to $\psi = \sigma^{*}f$ and 
$$ \cI (\tau^{\cV}(\sigma)) = \nabla_{E_i}^{c} \psi(E_i).$$
Projecting onto $\Skew \R^{2n}$ and $\R^{2n}$, respectively, the vertical tension field $\tau^{\cV}(\sigma) = \tr \nabla^{\cV} d^{\cV}\sigma $ gives rise to the two harmonic section equations of \cite{Vergara-Diaz2006}*{Theorem 3.2}:

\begin{equation}
\begin{cases}\label{tension field almost contact structures}
(\cI_1 \tau^{\cV}(\sigma))\hat{} &= -\tfrac{1}{4} [\bar{\nabla}^{*}\bar{\nabla} J ,J] \\
(\cI_2 \tau^{\cV}(\sigma))\hat{} &=  \nabla^*\nabla \xi + |\nabla \xi|^2 \xi - \tfrac{1}{2} J \tr \bar{\nabla}J\otimes\nabla\xi ,
\end{cases}
\end{equation}
still with the conventions that $\bar{\nabla}^{*}\bar{\nabla} J= \tr \bar{\nabla}^2 J$ and $\nabla^*\nabla \xi = \tr \nabla^2 \xi$. For the harmonic map equation, we use \cite{Wood2003} to obtain:
$$ 
\tfrac{1}{4} \langle R(E_i , X),J\bar{\nabla}_{E_i}J\rangle + \langle R(E_i , X)\xi, \nabla_{E_i} \xi\rangle =0 .
$$
The energy functional of an almost contact structure can then be computed to be
$$
E(\sigma) = \frac{\dim(M)}{2} + \frac{1}{2} \int_{M} \frac{1}{4} |\bar{\nabla} J|^2 + |\nabla\xi|^2 \, v_g
$$
and, since variations are considered among unit vector fields and $(1,1)$-tensors related by \eqref{eq1}, this could be seen as a constrained variational problem.

Examples of harmonic almost contact structures are rather numerous, eg. (i) the canonical structure of a hypersurface of a K{\"a}hler manifold is harmonic, whenever the Reeb vector field is harmonic; (ii) on the unit sphere $\sn^{2n-1} \subset\C^{n}$, it is moreover a harmonic map; (iii) and the latter indeed holds on any Sasakian manifold; (iv) among hyperspheres of the nearly K{\"a}hler $\mathbb S^6$, the equator is the only one to define a harmonic section, and it is also a harmonic map; (v) the property (iv) was further extended to all nearly cosymplectic manifolds, i.e. when $\nabla\phi$ is skew-symmetric (cf. \cite{Loubeau2017} for further details).

\subsection{The harmonic almost contact structure flow from \eqref{eq: HSF}}

In view of \eqref{torsion almost contact structure} and \eqref{tension field almost contact structures}, applying the canonical isomorphism $\cI$ to both sides of the harmonic section flow equation~\eqref{eq: HSF} yields the \emph{harmonic almost contact structure flow}:
\begin{equation}
\label{harmonic almost contact structure flow}\tag{HACtSF}
\begin{cases}
 \partial_t J &= \tfrac{1}{2} \Big( \bar{\nabla}^{*}\bar{\nabla} J + J (\bar{\nabla}^{*}\bar{\nabla} J )J \Big)  \\
\partial_t \xi &=  \nabla^*\nabla \xi + |\nabla \xi|^2 \xi - \tfrac{1}{2} J \tr \bar{\nabla}J \otimes\nabla\xi .
\end{cases},
\quad
\text{on }M_\rT.
\end{equation}
In this case, just as for parallelisms, the analysis of the corresponding harmonic flow seems hitherto uncharted, so Theorems  \ref{thm: uniqueness and s-t existence}-\ref{thm: HSF under bounded torsion} bear the following original consequence:
\begin{corollary}
\label{cor: bounded T => T-> infty ACTCS}
    Let $(M^{2n+1},g,\xi_0,\theta_0)$ be a  closed almost contact manifold. For fixed $g$, suppose the torsion 
    $\big( \tfrac12 J \nabla J, \nabla \xi \big)$, from \eqref{torsion almost contact structure}
    is uniformly $L^{4n+2}(M)$-bounded along the harmonic almost complex structure flow \eqref{harmonic almost contact structure flow}, with initial condition defined by $\sigma_0$ under the correspondence \eqref{eq: Xi correspondence}. Then the flow~\eqref{harmonic almost contact structure flow} 
    admits a \emph{continuous} solution $(\xi,J)(t)$ for all time. Moreover, there exists a strictly increasing sequence $\{t_j\}\subset \left[0,+\infty\right[$ 
    such that $(\xi,J)(t_j)\overset{C^\infty}{\longrightarrow} (\xi_\infty,J_\infty)\in \Gamma(TM)\times\End(TM)$, and the limiting almost contact structure satisfies the harmonicity condition:
    $$
    [ \bar{\nabla}^{*}\bar{\nabla} J_\infty,J_\infty]=0
    \qandq
     {\nabla}^{*}{\nabla} \xi_\infty + |\nabla \xi_\infty|^2 \xi_\infty - \tfrac{1}{2} J_\infty \tr \bar{\nabla}J_\infty\otimes\nabla\xi_\infty =0.
    $$
\end{corollary}

\section{Almost complex structures}

Let $(M^{2n},g)$ be an even-dimensional Riemannian manifold. A compatible \emph{almost complex structure} on $M$ is a $(1,1)$-tensor field $J$ such that $J^2 = -\id$ and $g(J. , J.) = g(.,.)$. This is one of the most studied and fruitful types of geometric structures, especially in the K{\"a}hler case. Nonetheless, many central questions remain unanswered, for example the existence of an integrable almost complex structure on (non-standard) $\sn^6$, or a topological characterisation of manifolds admitting almost complex structures.
The introduction of a variational problem will hopefully shed a new light on these objects. This case comes first, historically and in importance, and its study was pioneered by \cites{Wood1993,Wood1995}.

\subsection{Torsion  of an almost complex structure}

On an even-dimensional manifold, consider the classical situation of twistor spaces, in which $H=\mathrm{U}(n)\subset \mathrm{SO}(2n)$ is embedded by 
$$
A+\bi B \mapsto \begin{pmatrix}
A & -B  \\
B & A  
\end{pmatrix}.
$$
The Lie algebra $\fh=\fu(n)$ consists of anti-symmetric matrices commuting with
$$J_0 = 
\begin{pmatrix}
{\mathbb O}_{n} & - {\mathbb I}_{n}  \\
{\mathbb I}_{n} & {\mathbb O}_{n} 
\end{pmatrix},
$$
and $\fm$ anti-symmetric matrices anti-commuting with $J_0$. On the vector bundle $\pi^* TM \to N$, denote by $\Xi=\cJ$ the universal almost complex
structure introduced in Section~\ref{Section 1.2}.  At each $U(n)$-class of frames $y\in N$, the tensor ${\cJ}_y \in {\mathrm{End}}(T_{\pi(y)}M)$ is modelled on $J_0$, with respect to any orthonormal frame of $p^{-1}(y)$. Any element $\beta\in\fg=\so(2n)=\fu(n)\oplus\fm$ decomposes as:
$$ 
\tfrac{1}{2} J_0 [\beta , J_0] 
\oplus
- \tfrac{1}{2} J_0 \{\beta , J_0\},
$$
where $\{A,B\}= AB+BA$ is the anti-commutator.
Using ~\eqref{eq: diff} to differentiate the universal structure, then  $\{\fm,\cJ\}=0$ and $\cJ^2 = - \id$, we have
\begin{align*}
    \nabla_A \cJ 
    &= [ f(A),\cJ ] = -2 \cJ \circ f(A) \\
    \Leftrightarrow\quad
    f(A) &= \tfrac12 \cJ \circ \nabla_A \cJ.
\end{align*}
Applying this to  the pull-back $\sigma^* f$, 
we recover the torsion, as first obtained in \cite{Wood2003}*{Theorem 4.2}:
\begin{equation} 
\label{eq: torsion acs}
    \cI(d^{\cV}\sigma) = \cI(\sigma^* f) = \tfrac12 J \nabla J.
\end{equation}
One can easily check that the energy functional is
$$
E(\sigma) 
= \frac{\dim(M)}{2} + \frac{1}{2} \int_{M} \frac{1}{4} |\nabla J|^2 \, v_g ,
$$
so that K{\"a}hler structures are absolute minimisers.

\subsection{Harmonic almost complex structures}

In the case at hand, since $[\fm,\fm] \subset\fu(n)$ the connections $\nabla^c$ and $\nabla^{\ufm}$ of Lemma~\ref{lemma: 2nd ff} coincide (equivalently $G/H$ is a symmetric space). Moreover, since $\nabla^{\ufm}$ is the connection on tensors, the $\ufm$-component of the derivative of a section $\alpha$  of $\ufm$ is
$$
\nabla^c_{A} \alpha = \tfrac12 \cJ [\nabla_{A} \alpha , \cJ],
$$
where $\cJ$ is the universal almost complex structure defined above.
Taking $\alpha=\sigma^* f = \tfrac12 J \nabla J$ to be the torsion and tracing, we obtain the corresponding Euler-Lagrange equation \cite{Wood1993}:
$$ 
\cI(\tau^{\cV}(\sigma)) 
= -\frac{1}{4} [ \nabla^* \nabla J,J],
$$
with the convention that $\nabla^{*}\nabla J= \tr \nabla^2 J$.
So $J$ is a harmonic section if and only if it commutes with its rough Laplacian. Moreover, it is a harmonic map if  \cite{Wood2003}
$$ 
g( R(E_i , Z)J, \nabla_{E_i}J) =0, \quad \forall Z\in TM.
$$

The first examples of harmonic sections, but also of harmonic maps, have been nearly-K{\"a}hler structures \cite{Wood1993}, largely because of their interesting curvature identities \cite{Gray1969}. On the other hand, 
 $(1,2)$-symplectic structures are harmonic sections if, and only if, the operator $\Ric^* (=\tr R(X,. ) J.)$  is symmetric.
The Calabi-Eckmann complex structure on the product of odd-dimensional spheres and the Abbena-Thurston almost-K{\"a}hler structure are harmonic sections, but not critical for the volume functional. Davidov and Muskarov \cites{Davidov17,DavidovMushkarov2002} show the harmonicity, as sections or maps, of the Atiyah-Hitchin-Singer and Eells-Salamon almost Hermitian structures on the twistor space of an oriented Riemannian four-manifold.

\subsection{The harmonic almost complex structure flow from \eqref{eq: HSF}}

As for the others cases, the torsion equation~\eqref{eq: torsion acs} can be extended to the manifold $M_\rT$, thus allowing us to compute that $\cI (\partial_t \sigma_t)=\tfrac12 J \partial_t J$. Applying the canonical isomorphism~\eqref{eq: isomorphism I:V->m} to the flow~\eqref{eq: HSF}, we recover the \emph{harmonic almost complex structure flow} recently introduced 
in~\cite{He2019a}:
\begin{equation}
\label{harmonic almost complex structure flow}
\tag{HACxSF}
    \partial_t J 
    =  \tfrac12\Big( \nabla^* \nabla J 
    + J (\nabla^* \nabla J) J\Big)
\end{equation}
As an immediate consequence of Theorems  \ref{thm: uniqueness and s-t existence}-\ref{thm: HSF under bounded torsion}, we have therefore:
\begin{corollary}
\label{cor: bounded T => T-> infty ACS}
    Let $(M^{2n},g,J_0)$ be a closed almost Hermitian manifold. For fixed $g$, suppose the torsion \eqref{eq: torsion acs} is uniformly $L^{4n}(M)$-bounded along the harmonic almost complex structure flow~\eqref{harmonic almost complex structure flow}, with initial condition defined by $J_0=:\sigma_0^*\cJ$ under the correspondence \eqref{eq: Xi correspondence}. Then \eqref{harmonic almost complex structure flow} admits a \emph{continuous} solution $J(t)$ for all time. Moreover, there exists a strictly increasing sequence $\{t_j\}\subset \left[0,+\infty\right[$ 
    such that $J(t_j)\overset{C^\infty}{\longrightarrow} J_\infty\in\End(TM)$, and the limiting almost complex structure satisfies the harmonicity condition:
$$
[ {\nabla}^{*}{\nabla} J_\infty,J_\infty]=0.
$$
\end{corollary}

In~\cite{He2019a}, He and Li 
study the harmonic flow for almost complex structures and prove general short-time existence (Theorem 2.1) and long-time existence for initial data with small energy with convergence to a Kähler structure (Theorem 2). In~\cite{He2019}, weak almost complex structures are introduced, an $\epsilon$-regularity is proved and energy-minimising structures are studied.

\section{$\rG_2$-Structures on $7$-manifolds}

$\rG_2$-structures on an oriented and spin Riemannian manifold $(M^7,g)$ are reductions from $\SO(7)$ to $\rG_2$, i.e. principal $\rG_2$-subbundles of the frame bundle $p:\P\to M$. This is encoded by a $3$-form $\varphi \in \Lambda^3 T^\ast M$ which is modelled, pointwise, on the canonical vector cross-product $\varphi_0\in \Lambda^3(\R^7)^*$, which is stabilised by the standard action of $\rG_2$. Precise conditions for the existence of a $\rG_2$-structure are known to be purely topological:  orientability and spinability, cf.~\cite{Bryant2006} for Gray's argument. A $\rG_2$--structure induces a \emph{  $\rG_2$--metric}  $g_\varphi$, and its \emph{full torsion tensor} is $T:=\nabla^{g_\varphi}\varphi$; it is said to be \emph{closed} if  $d\varphi=0$ and \emph{coclosed} if $d\psi=0$, where $\psi:=\ast_\varphi\varphi$, and the torsion-free condition $\nabla \varphi=0$ is equivalent to being both at once, in which case $(M,\varphi)$ is said to be a \emph{$\rG_2$-manifold} and $\Hol(g_\varphi)\subset\rG_2$.   Distinct  $\rG_2$--structures are \emph{isometric} if they yield the same  $\rG_2$--metric, and the article \cite{Grigorian2017} interprets the \emph{divergence-free torsion} condition    $\Div T=0$  as a `gauge-fixing' among isometric $\rG_2$--structures.
Any $\rG_2$--structure can be deformed to become coclosed, but so far we only know that one may deform into a closed one if $M$ is an open manifold, cf.~\cite{Crowley2015}.

In the language of the present paper, the universal $\rG_2$-structure mediates the one-to-one correspondence between $\rG_2$-structures on $(M^7,g)$ and sections $\sigma$ of the fibre bundle $\pi : N:=P/\rG_2 \to M$, with fibre $\SO(7)/ \rG_2 \simeq \R \rP^7$. The principal $\rG_2$-bundle $q : \P \to N$ is such that $\pi \circ q =p$.
Moreover, $\pi$ is isomorphic to the associated bundle $\P \times_{\SO(7)} \SO(7)/ \rG_2$. That $\rG_2$-structures could be regarded as sections of an $\R P^7$-bundle was indeed already known to Bryant \cite[Remark 4]{Bryant2006}. 

\subsection{Torsion   of a $\rG_2$-structure}
Let $V= (\R^7)^\ast$ and $\varphi_0 \in \Lambda^3 V$, $\psi_0 = \ast \varphi_0 \in \Lambda^4 V$, be the standard $\rG_2$-structure on $\R^7$.
We identify the Lie algebras $\so(7)\simeq\Lambda^2 V = \fg_2 \oplus \fm$, with
\begin{align*}
    \fg_2 
    &= \left\{ \eta \in \Lambda^2 V : \ast (\eta \wedge \varphi_0) = \eta\right\},\\
    \fm 
    &= \left\{ \eta \in \Lambda^2 V : \ast (\eta \wedge \varphi_0) = -2 \eta\right\}.
\end{align*}
\begin{multicols}{2}
One easily checks that $\Ad_{\SO(7)}(\rG_2)\fm \subset \fm$, since the $\rG_2$-action preserves $\varphi_0$, so $\SO(7)/ \rG_2$ is reductive. 
Using the standard inner-product of $\SO(7)$, we have an invariant Riemannian metric on $\SO(7)/ \rG_2$ and an $\Ad_{\SO(7)}(\rG_2)$-invariant inner-product $\lan , \ran$ on $\fm$ and the canonical bundle of Section~\ref{sec: canonical geom on hfb} is now the vector bundle $\ufm = \P \times_{\rG_2} \fm \to N$ while the canonical isomorphism $\cI : \cV \to \ufm$ maps $q_{\ast} (a^\ast(z)) \in \cV$ to $z \bullet a = [ (z,a)]_{\rG_2}$.
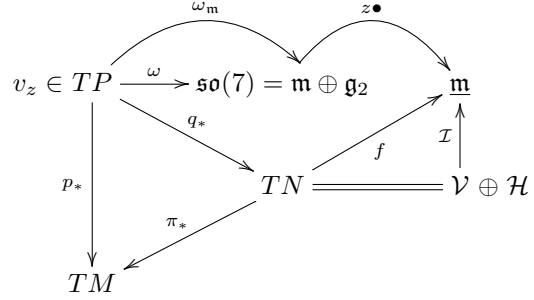
\columnbreak
$$\xymatrix{
\llap{$v_z\in$ }T\P 
	\ar[dr]^{q_*} 
    \ar[dd]_{p_*} 
    \ar[r]^-\omega 
    \ar@/^2pc/[r]!<5mm,0mm>^{\omega_{\fm}}
    & \so(7) = \fm \ar@/^2pc/[r]^{z\bullet} \oplus \fg_2 & {\ufm} \\ 
  &TN 
  	\ar[dl]_{\pi_*} 
    \ar@{=}[r]
    \ar[ur]_f 
    & \cV \rlap{ $\oplus\;\cH$}
  		\ar[u] ^\cI  \\ 
TM & &     
}$$
\end{multicols}

We define a Riemannian metric on $N$ by $h = \pi^\ast g + \lan f  , f  \ran$, so the map $\pi : (N,h) \to (M,g)$ becomes a Riemannian submersion, and the \emph{universal $\rG_2$-structure} assigns to each class of frames $y\in N$ the $3$-form in $( \Lambda^3 T^\ast M)_{\pi(y)}$ given by $\varphi_0$ in any frame of $q^{-1}(y)$:
$$
\Phi \in \Gamma(N, \pi^\ast ( \Lambda^3 T^\ast M)),
\quad
\Phi(y) = y^\ast \varphi_0.
$$ 
\begin{multicols}{2}
Since $\pi^\ast ( \Lambda^3 T^\ast M)\simeq\pi^\ast \P \times_{\SO(7)} \Lambda^3 V$, the \mbox{geometric} representation 
$
\rho : \pi^\ast ( \Lambda^3 T^\ast M) \to \Lambda^3 V
$ 
induces a $\SO(7)$-equivariant map
$$ 
\wtilde{\Phi} := \rho \circ \Phi \circ \pi^\ast p : \pi^\ast \P \to \Lambda^3 V,
$$
In view of the $\rG_2$-invariant embedding $\P \hookrightarrow \pi^\ast \P$, defined by $ z \mapsto (z, q(z))$, we have  $\wtilde{\Phi}|_\P = \varphi_0$.
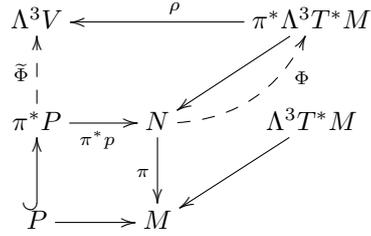
\columnbreak
$$\xymatrix{
     \Lambda^3 V 
	\ar@{<-^{}} [rr]^{\rho}
    &&\pi^\ast \Lambda^3 T^\ast M \ar[dl]\\
\pi^*\P 
	\ar[r]_{\pi^*p}
    \ar@{-->} [u]^{\wtilde{\Phi}}
    &N \ar[d]_\pi \ar@{-->}@/_1pc/[ur]_(0.75)\Phi& \Lambda^3 T^\ast M \ar[dl]\\
\P \ar@{^{(}->}[u] \ar[r]&M  &
}$$
\end{multicols}
Let $A\in TN$ and $E\in T\P$ a lift of $A$, i.e. $q^\ast (E) = A$, from Equation~\eqref{eq: diff}, we have
\begin{eqnarray} 
\nabla_A \Phi = f (A) . \Phi . \label{eq2}
\end{eqnarray}
In order to work with this formula, we need to understand the action of $\fm$ on $\Lambda^3 T^\ast M$,  resorting to the representation theory of $\rG_2$. Let us recall the well-known irreducible bundle splittings
\begin{equation}
\label{eq*}
    \wedge ^2 V = \wedge_{7}^2 \oplus \wedge_{14}^2,  
    \qwithq
    \begin{array}{rcl}
        \wedge_{7}^2 
        &=& \Big\{ \eta\in \wedge^2 : \ast ( \varphi_{0} \wedge \eta ) = -2 \eta\Big\} \\
        \wedge_{14}^2 
        &=& \Big\{ \eta\in \wedge^2 : \ast ( \varphi_{0} \wedge \eta ) = \eta \Big\}
    \end{array},
\end{equation}
and
\begin{equation} 
\label{eq**}
    \wedge^3 V = \wedge_{1}^3 \oplus \wedge_{7}^3 \oplus \wedge_{27}^3,
    \qwithq
   \begin{array}{rcl}
        \wedge_{1}^3 
        &=& \Big\{f \varphi_{0} : f\in C^{\infty}(\R^7) \Big\} \\
        \wedge_{7}^3 
        &=& \Big\{X \intprod \psi_{0}: X\in \Gamma(\R^7) \Big\}\\ 
        \wedge_{27}^3 
        &=& \Big\{h \intprod \varphi_{0} : h\in \odot^2 V, \tr h = 0\Big\}
   \end{array}.
\end{equation} 

The natural action of $\rG\rL (7,\R)$ on $\wedge^3 V$ descends to an action of $\so(7) \simeq \wedge^2 V$ on $\wedge^3 V$. The map 
$ \beta \mapsto \beta . \varphi_0$ has kernel isomorphic to the subspace $\wedge_{14}^2 \simeq \fg_2$, so for $\beta \in \wedge_{7}^2$ in its complement, 
there exists $X\in \Gamma(\R^7)$ such that $\beta = X \intprod \varphi_0$. On the other hand, $\beta . \varphi_0 \in \Lambda^{3}_{7}$, so $\beta . \varphi_0 = Y \intprod \psi_0$ for some $Y \in \Gamma(\R^7)$, and elementary representation theory shows that $Y = -3X$,  cf.~\cite{Karigiannis2007}.
Furthermore, 
$$ (X \intprod \psi_0) \intprod \psi_0 = - 24 X^b $$
hence, for $\beta \in \wedge_{7}^2$,
\begin{equation} 
\label{eq3}
    \left( (\beta.\varphi_0) \intprod \psi_0 \right) \intprod \varphi_0 = 72 \beta . 
\end{equation}
The existence of a $\rG_2$-structure implies a splitting of $\Lambda^2 T^\ast M$ and $\Lambda^3 T^\ast M$ according to \eqref{eq*} and \eqref{eq**}. Pulling back \eqref{eq3} by $\pi$ and combining with Equation~\eqref{eq2}, we find
$$
f (A) = \tfrac{1}{72} 
\left( 
    \nabla_{A} \Phi \intprod \Psi \right) 
\intprod \Phi ,
\qwithq
\Psi := \ast \Phi \in \pi^\ast \left( 
    \Lambda^4 T^\ast M 
\right).
$$

Let $\sigma : M \to N$ be a section of $\pi$ and $\varphi \in \Lambda^3 T^\ast M$ the corresponding $\rG_2$-structure. By definition, the pull-back by $\sigma$ of the homogeneous connection form gives the vertical component of $d\sigma$: 
$$
(\sigma^\ast f )(X) = f (d\sigma(X)) = \cI(d^\cV \sigma(X)), \quad \forall X\in TM.
$$
Plugging this into \eqref{eq3} yields
\begin{align*}
    f (d\sigma(X)) &= \tfrac{1}{72} \left( \left( \nabla_{d\sigma(X)} \Phi \intprod \Psi\right) \intprod \Phi\right) \circ\sigma \\
    &= \tfrac{1}{72} \left( (\nabla_{d\sigma(X)} \Phi) \circ\sigma\intprod \Psi \circ\sigma \right) \intprod \Phi \circ\sigma \\
    &= \tfrac{1}{72} \left( \nabla_{X} (\Phi \circ\sigma) \intprod \Psi \circ\sigma \right) \intprod \Phi \circ\sigma \\
    &= \tfrac{1}{72} \left( (\nabla_{X} \varphi) \intprod \psi \right) \intprod \varphi, 
\end{align*}
with $\psi = \ast\varphi \in \Lambda^4 T^\ast M$.
So
$$
\cI(d^\cV \sigma(X)) = \tfrac{1}{72} \left( (\nabla_{X} \varphi) \intprod \psi \right) \intprod \varphi,
$$
or indeed, since $\nabla_{X} \varphi \in \Lambda^{3}_{7}$, in terms of  the (full) torsion tensor $\nabla_{X} \varphi =: T(X) \intprod \psi$,
\begin{equation}
\label{eq: relation torsion}
    \cI(d^\cV \sigma(X)) = -\frac{1}{3} T(X) 
    \intprod \varphi.
\end{equation}

Knowing the homogeneous connection form, we can compute the vertical energy density of $\sigma : M \to N$:
\begin{align*}
    |d^\cV\sigma|^2 
    &= \sum_{i=1}^{7} |d^\cV\sigma (e_i)|^2
     = \sum_{i=1}^{7} h \left( (d^\cV\sigma) (e_i), (d^\cV\sigma) (e_i) \right) \\
    &= \sum_{i=1}^{7} \lan f (d^\cV\sigma (e_i)), f (d^\cV\sigma (e_i)) \ran 
     = \tfrac{1}{9} \sum_{i=1}^{7} \lan T(e_i) \intprod \varphi, T(e_i) \intprod \varphi \ran \\
    &= \tfrac{2}{3} |T|^2,
\end{align*}
since $\lan X \intprod \varphi, X \intprod \varphi \ran = 6 |X|^2$, by \cite[Lemma A8]{Karigiannis2007}.
Therefore (assuming $M$ compact),
\begin{equation}
\label{energy sigma}
    E (\sigma) = \tfrac{1}{2} \int_M |d^{\cV}\sigma|^2  
    = \tfrac{1}{3} \int_M |T|^2,
\end{equation}
and, as $\pi : (N,h) \to (M,g)$ is a Riemannian submersion, the (full) energy functional is
$$ 
\bar E (\sigma) = \tfrac{7}{2} + \tfrac{1}{3} \int_M |T|^2 \, v_g .$$

\subsection{Harmonicity of a $\rG_2$-structure}
To determine the vertical tension field of $\sigma : M \to N$, we take $X,Y \in TM$ and express 
$\cI\left( (\nabla^\cV d^\cV \sigma)(X,Y)\right)$, where $\nabla^\cV$ is the vertical component of the Levi-Civita connection $\nabla^N$ of $(N,h)$, in terms of the connection $\nabla^\omega$ on $\uso(7)$ inherited from the curvature form of $\omega$ on $P$.
It is easy to verify, using~\cite{Karigiannis2007}*{Proposition 2.5}, that the homogeneous space $\SO(7)/\rG_2$ is indeed naturally reductive (but not a symmetric space), i.e. geodesics are exactly given by orbits of the exponential map of $\SO(7)$. Therefore, by~\cite{Wood2003}*{Corollary 2.5}, the $\fm$-component of the structure equation for $\omega$ implies that $\nabla^\cV$ and the connection $\nabla^c$ on $\ufm$, induced by $\omega_{\fg_2}$ on $q : \P \to N$, are related by:
$$ 
\cI\left( \nabla^{\cV}_{A} V \right) = \nabla_{A}^{c} (\cI V) + \frac{1}{2} [ f  A, \cI V]_{\ufm} .
$$
However, pulling back by $\pi:N\to M$ the connection $\nabla^\omega$ on $\uso(7)$, it also restricts to a connection on $\ufm$, cf. \eqref{eq: canonical connection},
so the above relation implies
\begin{align*}
    &\cI\left( (\nabla^\cV d^\cV \sigma)(X,Y) \right) = 
    \cI\left( \nabla_{X}^{\sigma^{-1}v} ( (d^\cV\sigma)(Y)) - d^\cV\sigma (\nabla_X Y) \right) \\
    &= \nabla_{X}^{\sigma^{-1}c} ((\sigma^\ast f )(Y)) + \tfrac{1}{2} [ (\sigma^\ast f )(X), (\sigma^\ast f )(Y)]_{\ufm} 
- (\sigma^\ast f )(\nabla_X Y) \\
    &= \left( \nabla^{\omega} (\sigma^\ast f )\right) (X,Y) 
+ \tfrac{1}{2} [ (\sigma^\ast f )(X), (\sigma^\ast f )(Y)]_{\ufm} 
- [ (\sigma^\ast f )(X), (\sigma^\ast f )(Y)] .
\end{align*}
Taking traces, 
$$
\cI \left( \tau^\cV (\sigma)\right) = \tr \nabla^\omega (\sigma^\ast f ) .
$$

Identifying $\so(7)\simeq\Lambda^2 V$, the associated bundle is $\uso(7):=\P \times_{\SO(7)}~\Lambda^2 V\simeq \Lambda^2 T^\ast M$, and the connection $\nabla^\omega$ is exactly the standard one on $\Lambda^2 T^\ast M$, inherited from the Levi-Civita connection of $(M,g)$.
Furthermore, if $\{e_i \}_{i=1,\dots, 7}$ is an orthonormal frame field:
\begin{align*}
    \cI( \tau(\sigma)) 
    &= \sum_{i=1}^{7} \left( \nabla^{\omega} (\sigma^\ast f )(e_i , e_i ) \right) 
    = \sum_{i=1}^{7} \nabla_{e_i} \left( (\sigma^\ast f )(e_i) \right) - (\sigma^\ast f )(\nabla_{e_i} e_i) \\
    &= \sum_{i=1}^{7} \nabla_{e_i} \left( f (d\sigma (e_i)) \right) - f (d\sigma(\nabla_{e_i} e_i)) 
    = - \frac{1}{3} \sum_{i=1}^{7} \nabla_{e_i} \left( T(e_i) \intprod \varphi \right) - T(\nabla_{e_i} e_i) \intprod \varphi \\
    &= - \frac{1}{3} \sum_{i=1}^{7} \left( \nabla_{e_i} T \right)(e_i) \intprod \varphi - T(e_i) \intprod 
\nabla_{e_i}\varphi \\ &
    = - \frac{1}{3} \sum_{i=1}^{7} \left( \nabla_{e_i} T \right)(e_i) \intprod \varphi - T(e_i) \intprod (T(e_i) \intprod \psi) \\
    &= - \frac{1}{3} ( \di T) \intprod \varphi,
\end{align*}
by skew-symmetry of $\psi$.
So
\begin{equation} 
\label{HSE}
    \cI( \tau(\sigma)) = - \frac{1}{3} ( \di T ) \intprod \varphi .
\end{equation}

\begin{remark}
    As was done by Wood~\cite{Wood2003} and Vergara-Wood~\cite{Vergara-Diaz2006} for almost complex and almost contact structures, respectively, we characterise sections $\sigma : M \to N$ which are moreover harmonic maps, in addition to being harmonic sections. We follow the blue-print of Section~\ref{Section 1.4} and need to consider the $\fm$-component of $\Omega$, the curvature form of $\omega$, $\Omega_{\fm} = \tfrac{1}{3} \Omega  - \tfrac{1}{3} \ast \left( \varphi_0 \wedge \Omega\right)$, and observe that it is the restriction to $\P$ of the two-form $\tfrac{1}{3} \widetilde{\Omega}  - \tfrac{1}{3} \ast \left( \Phi  \wedge \widetilde{\Omega} \right)$ on $\pi^\ast \P$, $\widetilde{\Omega}$ being the pull-back of $\Omega$.
As, $\widetilde{\Omega} = \pi_{\ast} R$, on $N$, for vectors $A,B \in TN$:
$$
F (A,B) = \tfrac{1}{3} \left( \pi_{\ast} R  - \tfrac{1}{3} \ast \left( \Phi  \wedge \pi_{\ast} R \right)\right)(A,B) .
$$
Composing with a harmonic section $\sigma$ and applying $\pi\circ\sigma = \id$, we have
$$ (\sigma^\ast F ) (X,Y) = \tfrac{1}{3} \left( R  -  \ast (\varphi  \wedge R)\right)(X,Y).$$
Since $(\sigma^\ast f )(e_i) \in \ufm$,
\begin{align*}
    \lan \sigma^\ast f  , \sigma^\ast F  \ran (X) 
    &= \sum_{i=1}^{7} \lan (\sigma^\ast f )(e_i) , R_{\ufm}(e_i,X)  \ran 
    = \sum_{i=1}^{7} \lan (\sigma^\ast f )(e_i) , R(e_i,X)  \ran \\
    &= - \tfrac{1}{3} \sum_{i,j,k=1}^{7} (T(e_i) \intprod \varphi)( e_k , e_l) R(e_i , X)(e_k , e_l) ,
\end{align*}
and, taking $X= e_p$ in local coordinates, the harmonicity condition for $\sigma$ becomes
\begin{equation} 
\label{HM}
    -3 \lan \sigma^\ast f  , \sigma^\ast F \ran (e_p) 
    = 
    \sum_{i,j,k,l=1}^{7} T_{ij}\varphi_{jkl}R_{ipkl} = 0,
    \quad 
     p = 1, \dots,7. 
\end{equation}
The Bianchi-type identity \cite[Th. 4.2]{Karigiannis2007} states that
$$                                  \tfrac{1}{2} R_{ipkl}\varphi_{jkl} 
= \nabla_i T_{pj} - \nabla_p T_{ij} - T_{ik}T_{pl}\varphi_{klj},
$$
hence, by skew-symmetry of $\varphi$, \eqref{HM} becomes
\begin{align}
0 &= \sum T_{ij} \left( \nabla_i T_{pj} - \nabla_p T_{ij}\right) - T_{ij}T_{ik}T_{pl}\varphi_{klj} \notag \\
&= \sum T_{ij} \left( \nabla_i T_{pj} - \nabla_p T_{ij}\right). \label{eq4}
\end{align}

\end{remark}

\begin{example}

A $\rG_2$-structure is called \emph{nearly-$\rG_2$} if $(\nabla_X \varphi) (X) =0, \forall X \in TM$. Many examples can be found in \cite{Friedrich1997}, in particular the squashed seven-sphere, $\SU(3) / \mathbb{S}^1$, $\SO(5) / \SO (3)$ or principal $\mathbb{S}^1$-bundles over K{\"a}hler-Einstein manifolds, such as $\mathbb{S}^2 \times \mathbb{S}^2 \times \mathbb{S}^2$, $\C\rP^2 \times \mathbb{S}^2$, $\rF (1,2)$ or $\mathbb{S}^2$ times a del Pezzo surface.
\end{example}

\begin{corollary}[Nearly-$\rG_2$ structures as harmonic maps]
A nearly-$\rG_2$ structure defines a harmonic map from $(M,g)$ to $(N,h)$.
\end{corollary}
\begin{proof}
The nearly-$\rG_2$ condition can easily be shown to be equivalent to $T(X) = \lambda X$, for some $\lambda \in \R$. Then, immediately, $\di T =0$, so all nearly-$\rG_2$ structures are harmonic sections.

For the harmonicity condition, by \cite{Karigiannis2007}*{Corollary 4.4}, on a nearly-$\rG_2$ structure, the Bianchi-type identity becomes
$$ \sum_{a,b=1}^{7} R_{abji}\varphi_{abj} =0, \quad i = 1, \cdots,7,$$
since $T = \lambda \id_7$, so Equation~\eqref{HM} is satisfied (alternatively, one could use $\nabla T =0$ and the formulation of Equation~\eqref{eq4}).
\end{proof}


\subsection{The $\Div T$-flow of $\rG_2$-structures from \eqref{eq: HSF}}
\label{sec: Results for divT-flow}

Over the last decade, the field has incorporated an important analytic aspect in the use of natural geometric flows, outlined by the seminal works of Bryant \cite{Bryant2006} and Hitchin \cite{Hitchin2001a}, to produce special $\rG_2$--metrics with `the least possible torsion.
Moreover, Lotay and Wei developed the analytic theory for the Laplacian flow of closed $\rG_2$--structures, studied its soliton solutions, and proved long-time existence and stability results \cites{Lotay2017,Lotay2019,Lotay2019a}. The Laplacian flow and coflow are manifestations of the gradient flow of Hitchin's volume functional \cite{Hitchin2001a}, according to the initial value's irreducible decomposition.
While both should optimistically converge to $\rG_2$--structures with less torsion (ideally torsion-free), there occurs
a trade-off between the abundance of coclosed structures, relative to closed ones, and the failure of parabolicity of the coflow, which hinders important
properties such as even short-time existence. In some sense, restoring parabolicity of the coflow is equivalent to being able to solve $\Div T=0$ in each isometric class \cite{Grigorian2013,Grigorian2017}. While it is not reasonable to expect such a problem to be solvable in general, we believe our theory sheds new light onto this problem, by identifying the divergence-free torsion condition as actual harmonicity with respect to the Dirichlet energy, and indeed the isometric $\rG_2$-flow (see below) as the harmonic map heat flow of the associated $\SO(7)$-equivariant maps.

Let $(M,\varphi_0)$ be a  closed $7$-manifold with $\rG_2$-structure. The divergence-free torsion condition from \eqref{eq: relation torsion} and \eqref{HSE}  motivates a natural geometric flow, which evolves $\varphi_0$ along isometric $\rG_2$-structures:
\begin{equation}
\label{eq: div flow}
\left\{
\begin{array}{rcl}
\partial_t \varphi &=& (\Div T)\lrcorner \psi  \\
    \varphi(0) &=& \varphi_0 
\end{array}
\right.
\quad\text{on}\quad 
M_{\rT}:=M\times \left[0,\rT\right]. 
\end{equation}
Of course, up to conventions, this is literally \eqref{eq: HSF} for $H=\rG_2\subset\SO(7)$.

From its inception by Grigorian \cite{Grigorian2017}, the so-called \emph{isometric $\rG_2$-flow}, or `$\Div T$-flow', \eqref{eq: div flow} has attracted interest and triggered some rapid developments \cites{Bagaglini2019,Grigorian2019,Dwivedi2019}, see also \cite{Chen2018}.
Let us briefly relate some of their interesting results to the general theory of the harmonic section flow.
For uniqueness and short-time existence, our Theorem \ref{thm: uniqueness and s-t existence} generalises
    \cite{Bagaglini2019}, 
    \cite[Theorem 5.1]{Grigorian2019}, 
and \cite[Theorem 2.12]{Dwivedi2019}.
We know immediately, from Lemma \ref{lemma: E(t) unif bounded}, that the torsion $T$ remains uniformly $L^2$-bounded along such solutions, cf. \cite[Lemma 5.3]{Grigorian2019} and \cite[Proposition 2.5]{Dwivedi2019}.  
Further regularity of solutions is then inferred from Proposition \ref{prop: kappa -> 0}, for so long as $|T(t)|$ remains bounded, cf. \cite[Theorems 5.7 \& 5.8]{Grigorian2019}, and \cite[Theorem 3.7]{Dwivedi2019}, and  subsequential convergence to a smooth harmonic limit follows from Corollary \ref{cor: subseq sol to HSF under C^0-bound} and Proposition \ref{prop: C^0 bounds from L^p for subsolutions}: 

\begin{corollary}
\label{cor: bounded T => T-> infty G2}
    Let $(M,\varphi_0)$ be a  closed $7$-manifold with $\rG_2$-structure. Suppose the torsion \eqref{eq: relation torsion} is uniformly $L^{14}(M)$-bounded along the harmonic section flow \eqref{eq: HSF}, with initial condition defined by $\varphi_0=:\sigma_0^*\Phi$ under the correspondence \eqref{eq: Xi correspondence}. Then the isometric $\rG_2$-flow (\ref{eq: div flow}) admits a \emph{continuous} solution $\varphi(t)$ for all time. Moreover, there exists a strictly increasing sequence $\{t_j\}\subset \left[0,+\infty\right[$ 
    such that $\varphi(t_j)\overset{C^\infty}{\longrightarrow} \varphi_\infty\in\Omega_+^3(M)$, and the limiting $\rG_2$-structure is harmonic, which is equivalent to:
    $$
    \Div T_\infty=0.
    $$
\end{corollary}

Bounded torsion is a reasonable assumption for short time, in view of the growth estimate of Theorem \ref{thm: Tmax and doubling time}--(i), and indeed a necessary and sufficient condition for long-time existence, by Theorem \ref{thm: Tmax and doubling time}--(ii), cf. 
\cite[Theorem 5.4]{Grigorian2019}, and \cite[Theorem 3.8]{Dwivedi2019}.

\newpage
\section*{Afterword: further developments}

We sketch some interesting outcomes of the proposed theory of harmonic section flows, to be addressed in subsequent articles, besides of course working out the particular manifestations of \eqref{eq: HSF} for the reader's favourite normal reductive pair $H\subset G$.

\medskip
\noindent\textbf{$\varepsilon$-regularity: }
One important set of questions regarding sufficient conditions for long-term existence and convergence remains untreated in the present paper, namely whether a small `entropy' condition suffices to ensure eventually $C^0$-bounded torsion and hence convergence to a harmonic (or even torsion-free)  limit. Comparing for instance with what is known for $\rG_2$-strucures, eg. \cite[Theorems 6.1 \& 7.5]{Grigorian2019}
and \cite[Theorems 5.3 \& 5.7]{Dwivedi2019}, the \eqref{eq: HSF} then manifests itself as the isometric `$\Div T$-flow' and the methods from \cite{Boling2017} can be adapted to that specific context. Very similar results can be found also for almost complex structures in \cites{He2019,He2019a}. But we now see that these flows are closely related to the \eqref{eq: HMHF} for $G$-equivariant sections $s_t$ under the isometry $\mu$, so analogous formulations of monotonicity and $\epsilon$-regularity should hold in general for equivariant maps $s_t$ onto say homogeneous spaces. One should then be able to study the formation of singularities, generalising eg. \cite[\S6]{Dwivedi2019}. This next step will require a theory along the lines of \cite{Chen-Ding1990} for equivariant maps, which, to our best knowledge, has not yet been developed in the literature.

\medskip
\noindent\textbf{Harmonic homogeneous geometric structures: }
In view of Remark \ref{rem: homogeneous strs and Lauret}, in the particular context of \emph{homogeneous}  geometric structures on $M=K/L$, Corollary \ref{cor: bounded T => T-> infty G2} gives somewhat automatically the long-time existence and regularity of the flow. This could lead to an interesting classification programme based on the Gromov-Hausdorff limits by \cites{Lauret2012,Lauret2016}, especially in cases in which the existence of (non-flat) torsion-free geometric structures is known to be obstructed.

For instance, in the particular context of \emph{homogeneous $\rG_2$-structures}, torsion-free structures would be Ricci-flat and hence downright flat. On the other hand, coclosed structures are always harmonic, and compact homogeneous $7$-manifolds admitting such structures have been completely classified, independently by \cites{Reidegeld2010,Le2012}. A finer question would then be whether these spaces admit homogeneous harmonic $\rG_2$-structures which are not coclosed, which could be addressed by examining closely the set $\Crit(E)$. A complementary question would be whether those spaces which do not possess homogeneous coclosed structures admit any harmonic ones, hence could be classified by the limits of the harmonic section flow.



\bibliography{Bibliografia-2021-10}

\end{document}